\documentclass[11pt,a4paper]{article}

\usepackage[english]{babel}
\usepackage[utf8]{inputenc}
\usepackage{amssymb}
\usepackage{amsfonts}
\usepackage{amsmath}
\usepackage{amsthm}
\usepackage{graphicx}
\usepackage[colorinlistoftodos]{todonotes}
\usepackage{authblk}
\usepackage{cite} 
\usepackage[labelfont=bf,labelsep=space]{caption}
\usepackage[framed,numbered,autolinebreaks,useliterate]{mcode}
\usepackage{subfig}

\newtheorem{dfn}{Definition}
\newtheorem{thm}{Theorem}
\newtheorem{cor}{Corollary}
\newtheorem{rem}{Remark}

\begin{document}

\title{Entrainment and Synchronization in Heterogeneous Networks of Haken-Kelso-Bunz (HKB) Oscillators}

\author[1]{Francesco Alderisio}
\author[2]{Benoit G.~Bardy}
\author[1,3]{Mario di~Bernardo}

\affil[1]{Department of Engineering Mathematics, University of Bristol, Merchant Venturers' Building, BS8 1UB, United Kingdom}
\affil[2]{EuroMov, Montpellier University, 700 Avenue du Pic Saint-Loup, 34090 Montpellier, France.}
\affil[3]{Department of Electrical Engineering and Information Technology, University of Naples Federico II, 80125 Naples, Italy}
\maketitle

\begin{abstract}
In this paper we consider a heterogeneous network of Haken-Kelso-Bunz (HKB) nonlinear oscillators coupled through both linear and nonlinear interaction protocols. While some work exists on a system made up of only two nonlinearly coupled HKB oscillators as a model of human dynamics during interpersonal coordination tasks, the problem of considering a network of three or more HKBs has not been fully investigated. The aim of our work is to study convergence and synchronization in networks of HKB oscillators as a paradigm of coordination in multiplayer games. Convergence results are obtained under the assumption that the network is connected, simple and undirected. Analytical results are obtained to prove convergence when the oscillators are coupled diffusively. All theoretical results are illustrated via numerical examples. Finally, the effects of adding an external entrainment signal to all the agents in the network are analyzed and a model to account for them is proposed.
\end{abstract}

\section{Introduction}
Interpersonal coordination and synchronization between the motion of two individuals have been extensively studied over the past few decades \cite{HKB85,KKSS87, ST94, VMLB11}. 

Synergetic movements of two or more people mirroring each other frequently occur in many activities such as handling objects, manipulating a common workpiece, dancing, choir singing and movement therapy \cite{HT09, VOD10, MLVGSH12, LMH13, RP13}. It is of great importance to reveal not only the effects of mirroring movements among people on human physiological and mental functions, but also to deeply understand the link between intrapersonal and interpersonal coordination.
In social psychology, it has been shown that people prefer to team up with others possessing similar morphological and behavioral features, and that they tend to coordinate their movement unconsciously \cite{F82, LS11}. Moreover, much evidence suggests that motor processes caused by interpersonal coordination are strictly related to mental connectedness. In particular, motor coordination between two human subjects contributes to social attachment particularly when the kinematic features of their movement share similar patterns \cite{WH09,SRdBTA14}. 

In order to explain the experimental observations of human interpersonal coordination, mathematical models are usually derived to capture the key features of the observed behavior.
A classical example is the so-called HKB oscillator which was introduced in \cite{HKB85} to explain the transition from phase to antiphase synchronization in bimanual coordination experiments (for more details see \cite{KSSH87, JFK98}). The HKB nonlinear oscillator was shown to be able to capture many features of human coordination even beyond the bimanual synchronization experiments it was derived to explain. For example, HKB oscillators were used in \cite{ZATdB14_1,ZATdB14_2} in the context of the \emph{mirror game} \cite{NDA11}, often presented as an important paradigmatic case of study to investigate the onset of social motor coordination between two players imitating each other's hand movements.
Furthermore, in \cite{KdGRT09} the authors take inspiration from the \emph{dynamic clamp} of cellular and computational neuroscience in order to probe essential properties of human social coordination by reciprocally coupling human subjects to a computationally implemented model of themselves (HKB oscillator), referred to as Virtual Player (VP). Such concept, namely the \emph{human dynamic clamp} (HDC), was further investigated and developed in \cite{DdGTK14} in order to cover a broader repertoire of human behavior, including rhythmic and discrete movements, adaptation to changes of pacing, and behavioral skill learning as specified by the VP.
Besides, HKB oscillators were also used in \cite{ST94} in order to capture the rhythmic coordination between two individuals swinging hand-held pendulums, in \cite{VMLB11} in order to model spontaneous interpersonal postural coordination between two human people and account for the competition between the coupling to a visual target to track and the coupling to the partner, in \cite{RMIGS07} in order to qualitatively explain interpersonal movement synchronization between two human beings involved in rhythmic paradigms, and in \cite{AST95} in order to account for the frequency detuning of the phase entrainment dynamics of two people involved in interlimb coordination.

While coordination of two human players has been studied in numerous previous investigations, the case of multiple human players has been seldom studied in the existing literature, due to a combination of practical problems in running the experiments and lack of a formal method able not only to model the considered scenario but also to quantify and characterize the synchronization level of the ensemble.
Multiplayer games involve a group of three or more people engaged in a communal coordination task. The variety of scenarios that can be considered is vast due to the countless activities the players might be involved in (limb movements, finger movements, head movements, walking in a crowd, or more in general music and sport activities), the many ways in which the participants can interact and communicate with each other and the different ways all the players can be physically located with respect to each other while performing the specified task.

Some of the existing works on coordination of multiple human players include studies on choir singers during a concert \cite{HT09}, rhythmic activities as for example "the cup game" and marching tasks \cite{IR15}, rocking chairs \cite{FR10, RGFGM12} and coordination of rowers' movements during a race \cite{WW95}.
In these papers the authors provide several experimental results in order to analyze the behavior of a group of people performing some coordinated activities, but a rigorous mathematical model capable of capturing the observed results and explaining the features of the movement coordination among them is still missing. In particular, in \cite{FR10} the authors study both unintentional and intentional coordination by asking the players to try and synchronize the oscillations of the rocking chairs with their eyes shut or open. Synchronization is observed to spontaneously emerge when players observe each other's movements. Another study in which multiplayer activities are analyzed but a mathematical model is missing is carried out in \cite{YY11}: the authors use the symmetric Hopf bifurcation theory, which is a model-independent approach based on coupled oscillators, to investigate the synchronized patterns of three people during sport activities.

Further results about multiplayer activities deal with spontaneous group synchronization of arm movements and respiratory rhythms. For example, in \cite{CBVB14} the authors test whether pre-assigned arm movements performed in a group setting spontaneously synchronize and whether synchronization extends to heart and respiratory rhythms. In their study no explicit directions are given on whether or how the arm swingings are to be synchronized among participants, and experiments are repeated with and without external cues. Interestingly, when an external auditory rhythm is present, both motor and respiratory synchronization is found to be enhanced among the group. Also the overall coordination level is observed to increase when compared to that detected when the same experiments are again carried out in its absence.

While in \cite{CBVB14} no mathematical model is presented, in \cite{GDSC96} the effects of an external visual signal, when considering the postural sway of a human being, is explicitly modeled by introducing a linear oscillator with adaptive parameters depending on the frequency of the exogenous visual signal itself. 
Moreover, when considering the social postural coordination between two human beings \cite{VMLB11}, another model for an external visual stimulus is presented by adding a unidirectional coupling term to the dynamics of a nonlinear oscillator.

The main objective of this paper is to propose and analyze a model able to account for the onset of movement synchronization in multiplayer scenarios and explain some of the features observed experimentally in the existing literature.
Specifically, we consider a heterogeneous network of HKB nonlinear oscillators as a good model of multiplayer coordination and, as already done in \cite{MLVGSH12} for the case of two agents only, we regard it as a synchronization problem.
Each equation is used to model the movement of a different player and is therefore characterized by a different set of parameters to account for human-to-human variability. The effect of different interaction models, linear and nonlinear, will be investigated to understand under what conditions synchronization is observed to emerge.
Our analysis suggests that bounded synchronization is indeed a common emergent property in these networks whose occurrence can also be accounted for analytically in a number of different scenarios. Also, as expected from existing theoretical results, we find that the structure of the interactions among players has an effect on the coordination level detected in the network. 

Furthermore, the effects of adding an external sinusoidal signal will be studied in order to understand whether synchronization can be improved by means of an entrainment signal \cite{RdBS10}. 
Our analysis suggests that the synchronization level of the ensemble can indeed increase when the oscillation frequency of the external signal is similar to the natural angular velocity of the agents in the network. However, in all the other cases, the external signal acts as a disturbance and leads to a decrease of the coordination among the agents. 

We wish to emphasize that the study reported in this paper will form the basis of future experimental investigations which are currently being planned.

The paper is organized as follows.
In Sect. \ref{sec:preliminaries} some notation that shall be used in later sections is introduced.
In Sect. \ref{sec:problemstatement} the equation that describes the network is presented, in terms of both internal dynamics of each agent and coupling protocol thanks to which they can interact with each other.
In Sect. \ref{sec:synchmetr} some metrics are introduced to characterize the quality and the level of coordination in human groups.
In Sect. \ref{sec:testbed} a testbed scenario of multiplayer coordination in networks of human people is presented, while in Sect. \ref{sec:modres} the key synchronization features experimentally observed are reproduced by considering a heterogeneous network of HKB oscillators, and the effects of three different coupling strategies thanks to which they are interconnected are explored.
In Sect. \ref{sec:entrainment} the effects of adding an external entrainment signal is analyzed with respect to the overall synchronization level of the network.
In Sect. \ref{sec:mainresults} global bounded synchronization of the network when its nodes are connected through a linear diffusive coupling protocol is analytically proven to be achieved, and some numerical examples are provided in order to both illustrate the effectiveness of our analysis and to show that bounded synchronization can be achieved also when considering different couplings.
Finally, in Sect. \ref{sec:conclusion} a summary of our results and some possible future developments are presented.

\section{Preliminaries and background}
\label{sec:preliminaries}
We denote with $\otimes$ the Kronecker product between two matrices.
The operator $\lambda_k \left( \cdot \right)$ defined over a matrix indicates the $k$-th eigenvalue of the matrix itself, and $\lambda_M \left( \cdot \right)$ indicates its maximum eigenvalue when the matrix is real and symmetric and as a consequence all the eigenvalue are real as well.

A \emph{graph} is a tuple $\mathcal{G} = \{ \mathcal{V}, \mathcal{A} \}$ defined by a set of nodes $\mathcal{V} = \{ 1,...,N \}$ and a set of edges $\mathcal{A}  \subseteq \mathcal{V} \times \mathcal{V}$. 
A graph is said to be \emph{undirected} if $(i,j) \in \mathcal{A} \iff (j,i) \in \mathcal{A}$. In an undirected graph, two nodes $i$ and $j$ are said to be \emph{neighbors} if $(i,j) \in \mathcal{A}$. 
The matrix $A=\{a_{ij} \} \in \mathbb{R}^{N \times N}$, where

\begin{equation*}
a_{ij}\begin{cases}
>0, & \mbox{if } (i,j) \mbox{ are neighbors} \\ 
=0, & \mbox{otherwise}
\end{cases}
\end{equation*}
is called \emph{adjacency matrix}, and $a_{ij} \ge 0$ is called strength of the interaction between the pair $(i, j)$. In particular, a graph is said to be \emph{unweighted} if the interaction between two neighbors is equal to $1$.
A \emph{path} between nodes $h$ and $k$ is a sequence of nodes, with $h$ and $k$ as endpoints, such that every two consecutive nodes are neighbors. 
A graph is said to be \emph{simple} if $a_{ii} = 0 \ \forall i \in \mathcal{V}$, while it is said to be \emph{connected} if there exists a path between any two of its nodes. 
The matrix $L = \{l_{ij} \} \in \mathbb{R}^{N \times N}$ defined as

\begin{equation}
l_{ij} : =
\begin{cases}
\sum_{k=1}^{N} a_{ik}, & \mbox{if } i=j \\ 
-a_{ij}, & \mbox{if } i \neq j
\end{cases}
\end{equation}
is called \emph{Laplacian matrix} of the graph (or simply \emph{Laplacian}).
The Laplacian of any simple undirected graph is symmetric with zero row sum and is a positive semidefinite matrix with as many null eigenvalues as there are components in the graph. In particular, a connected graph has only one null eigenvalue.

Throughout the paper we shall consider a connected simple undirected network of $N$ agents assuming that any two players interact symmetrically with one another.

Before analyzing a multiplayer scenario, it is worth considering the simpler case of only two human players interacting with each other. The system that can be used to model the interaction between them is described as follows \cite{RMIGS07, FJ08}:

\begin{equation}
\begin{cases}
\ddot{x_1}+\left(\alpha x_1^2 + \beta \dot{x}_1^2 -\gamma \right) \dot{x}_1 + \omega_1^2 x_1 = I(x_1,x_2) \\
\ddot{x_2}+\left(\alpha x_2^2 + \beta \dot{x}_2^2 -\gamma \right) \dot{x}_2 + \omega_2^2 x_2 = I(x_2,x_1) \end{cases}
\end{equation}
where $x_i \in \mathbb{R}$ denotes the position of the $i$-th player, with $i=1,2$. The right-hand side of both equations represents the coupling term between the two players: in particular

\begin{equation}
I(w,z) := [ a+b\left( w-z \right)^2 ] \left( \dot{w}-\dot{z} \right)
\end{equation}

The term $\left(\alpha x_i^2 + \beta \dot{x}_i^2 -\gamma \right) \dot{x}_i$ represents the nonlinear damping of the oscillatory movement of player $i$. Specifically, the sign of $\gamma$ determines whether, in the absence of coupling, the oscillation is persistent ($\gamma>0$) or vanishes (vice versa) as time goes by: it is trivial to verify this by studying the stability of the origin and checking the sign of the eigenvalues of the Jacobian of the system. Moreover, $\alpha$ and $\beta$ determine the amplitude of such oscillation, while $\omega_i$ is related to its frequency. It has been proven that this model of two nonlinearly coupled oscillators account for the features observed during experimental data in bimanual experiments (see \cite{HKB85} for further details).

\section{Human to human coordination as a synchronization problem}
\label{sec:problemstatement}
In the introduction of this paper we have pointed out that the dynamics of two coupled HKB oscillators has been used to describe different kinds of interpersonal coordination tasks between two people, including bimanual coordination experiments, mirror game, social postural coordination and rocking chairs. According to the particular scenario considered, the state vector of each oscillator is used to represent position and velocity of the particular body part of interest of either of the players (finger, hand, head, and so forth). Following the same approach, we can consider a scenario in which more than two human beings are performing a multiplayer coordination task, as for example arm or hand rhythmic movements, rocking chairs, head tracking of a visual target and so on. In these cases, the state vector of each node represents position and velocity of the particular body part of interest of each player.
Therefore, the dynamics of each player when moving in isolation will be described by an HKB equation:

\begin{equation}
\label{eqn:hkbInternalDynamics}
f_i(t,x_i)=
\begin{bmatrix}
x_{i_2} \\
- (\alpha_i x_{i_1}^2+\beta_i x_{i_2}^2-\gamma_i)x_{i_2} - \omega_i^2 x_{i_1}
\end{bmatrix}
\end{equation}
where $x_i=[x_{i_1} \ x_{i_2}]^T \in \mathbb{R}^2$ is the state vector, with $x_{i_1}, x_{i_2}$ representing position and velocity of the $i$-th human player, respectively.

To model the interaction between different players we assume that the dynamics of each of them is affected by some coupling function $u_i$ which depends on the state of its neighbors.
In what follows we will explore the effects of three possible selections for such a function.
In particular, we are interested in analyzing the differences of the results provided by all of them and understanding which one leads to synchronization features which are the closest to the ones observed in some existing work about group synchronization of networks of several human people involved in a coordination task \cite{RGFGM12}.

\begin{enumerate}
\item \emph{Full state coupling}. With this kind of coupling, we assume that players adjust both their velocities and accelerations proportionally to the average mismatch between theirs and those of their neighbors. Mathematically, we have:

\begin{equation}
\label{eqn:gsldc}
u_i = -\frac{c}{\mathcal{N}_i} \sum_{j=1}^{N} a_{ij} \left( x_i-x_j \right)
\end{equation}


In particular, $\mathcal{N}_i>0$ is the number of neighbors of node $i$, while $c>0$ is the coupling strength among the agents.

\item \emph{Partial state coupling}. Next, we explore the case where players only adjust their accelerations according to the position and velocity mismatches from their neighbors:

\begin{equation}
\label{eqn:gsipvc}
u_i = - \begin{bmatrix}
0 \\
\sum_{j=1}^{N} \frac{a_{ij}}{\mathcal{N}_i} \left[ c_1 \left( x_{i_1}-x_{j_1} \right) + c_2 \left( x_{i_2}-x_{j_2} \right) \right]
\end{bmatrix}
\end{equation}


In particular, $\mathcal{N}_i>0$ is the number of neighbors of node $i$, while $c_1,c_2>0$ represent the position and the velocity coupling strengths, respectively.

\item \emph{HKB coupling}. Finally we consider an interaction model which is the direct extension to multiplayer coordination problems of the interaction function used in the classical HKB set up to model coordination between two players \cite{HKB85,FJHK96}. Specifically we choose the following nonlinear function:

\begin{equation}
\label{eqn:gshkbc}
u_i = \begin{bmatrix}
0 \\
\frac{c}{\mathcal{N}_i} \sum_{j=1}^{N} a_{ij} [a+b(x_{i_1}-x_{j_1})^2](x_{i_2}-x_{j_2})
\end{bmatrix}
\end{equation}


Once again, $\mathcal{N}_i>0$ is the number of neighbors of node $i$, while $c>0$ represents the coupling strength among the agents. 
\end{enumerate}

The resulting network model describing the interaction of a group of $N$ players can then be written as

\begin{equation}
\label{eqn:networkeq}
\dot{x}_i(t) = \begin{bmatrix}
x_{i_2} \\
- (\alpha_i x_{i_1}^2+\beta_i x_{i_2}^2-\gamma_i)x_{i_2} - \omega_i^2 x_{i_1}
\end{bmatrix} + u_i(t) \ \in \mathbb{R}^2
\end{equation}
where the coupling function $u_i$ can be chosen as one of those listed above.
We now explore under what conditions coordination, and hence synchronization, emerges for each of the three scenarios of interest. 

We wish to emphasize that, since the node parameters are heterogeneous, complete synchronization as defined in \cite{LDCH10} cannot be achieved. We will consider instead the case where bounded synchronization, as defined below, emerges.
Namely, we define the average trajectory as

\begin{equation}
\bar{x}(t) : = \frac{1}{N} \sum_{j=1}^{N} x_j(t)
\end{equation}
and the tracking error as

\begin{equation}
e_i(t) : = x_i(t)-\bar{x}(t) \quad \forall t \ge 0, i=1,...,N
\end{equation}

We also define the parameters vector for each node $i$ as $\vartheta_i := [\alpha_i \ \beta_i \ \gamma_i \ \omega_i]^T \in \mathbb{R}^4$, and we introduce the stack vectors $x(t) : = [x_1(t)^T \ x_2(t)^T \ ... \ x_N(t)^T]^T \in \mathbb{R}^{2N}$ and $e(t) : = [e_1(t)^T \ e_2(t)^T \ ... \ e_N(t)^T]^T \in \mathbb{R}^{2N}$ and the error norm $\eta(t) : = ||e(t)|| \in \mathbb{R}, \forall t \ge 0$, where $|| \cdot ||$ indicates the Euclidean norm.

We say that a network of HKB oscillators achieves coordination if and only if 

\begin{equation}
\label{eqn:gbseqdef}
\lim_{t \to \infty} \eta(t) \le \epsilon
\end{equation}
for any initial condition $x_{i,0}$ and parameter vector $\vartheta_i$ of the nodes in the network, where $\epsilon>0$ is a sufficiently small constant.

\section{Synchronization metrics}
\label{sec:synchmetr}
In order to quantify and analyze the synchronization level in a network of more than two agents, we use the metrics introduced in \cite{RGFGM12} to characterize the quality and the level of coordination in human groups. 

Let $x_k(t) \in \mathbb{R} \  \forall t \in [0,T]$ be the continuous time series representing the motion of each agent, with $k \in [1,N]$, where $N$ is the number of individuals and $T$ is the duration of the experiment. Let $x_k(t_i) \in \mathbb{R}$, with $k \in [1,N]$ and $ i \in [1,N_T]$, be the respective discrete time series of the $k$-th agent, obtained after sampling $x_k(t)$, where $N_T$ is the number of time steps and $\Delta T := \frac{T}{N_T}$ is the sampling period. Let $\theta_k(t) \in [-\pi,\pi]$ be the phase of the $k$-th agent, which can be estimate by making use of the Hilbert transform of the signal $x_k(t)$ \cite{KCRPM08}. We define the \emph{cluster phase} or \emph{Kuramoto order parameter}, both in its complex form $q'(t) \in \mathbb{C}$ and in its real form $q(t) \in [-\pi,\pi]$ as

\begin{equation}
q'(t) :=  \frac{1}{N} \sum_{k=1}^{N} e^{ \{ j \theta_k(t) \} }
\end{equation}

\begin{equation}
q(t) := {\rm atan2} \left(\Im(q'(t)),\Re(q'(t))  \right)
\end{equation}
which can be regarded as the average phase of the group at time $t$.

Let $\phi_k(t) := \theta_k(t) - q(t)$ be the relative phase between the $k$-th participant and the group phase at time $t$. We can define the relative phase between the $k$-th participant and the group averaged over the time interval $[t_1,t_{N_T}]$, both in its complex form $\bar{\phi}'_k \in \mathbb{C}$ and in its real form $\bar{\phi}_k \in [-\pi,\pi]$ as

\begin{equation}
\bar{\phi}'_k := \frac{1}{T} \int_{0}^{T} e^{ \{ j \phi_k(t) \} } \ dt \simeq \frac{1}{N_T} \sum_{i=1}^{N_T} e^{ \{ j \phi_k(t_i) \} }
\end{equation}

\begin{equation}
\qquad  \bar{\phi}_k := {\rm atan2} \left( \Im(\bar{\phi}'_k), \Re(\bar{\phi}'_k) \right)
\end{equation}

In order to quantify the degree of synchronization for the $k$-th agent within the group we define the following parameter

\begin{equation}
\label{eqn:r1}
\rho_k := |\bar{\phi}'_k| \quad \in [0,1]
\end{equation}
which simply gives information on how much the $k$-th agent is synchronized with the average trend of the group. The closer $\rho_k$ is to $1$, the better the synchronization of the $k$-th agent itself.
 
In order to quantify the synchronization level of the entire group at time $t$ we define the following parter

\begin{equation}
\label{eqn:r2}
\rho_{g}(t) := \frac{1}{N} \left | \sum_{k=1}^{N} e^{ \{ j [ \phi_k(t)- \bar{\phi}_k ] \} } \right | \quad \in [0,1]
\end{equation}
which simply represents the group synchronization: the closer $\rho_{g}(t)$ is to $1$, the better the synchronization level of the group at time $t$. 
Its value can be averaged over the whole time interval $[0,T]$ in order to have an estimate of the mean synchronization level of the group during the total duration of the performance:

\begin{equation}
\label{eqn:r3}
\rho_g := \frac{1}{T} \int_{0}^{T} \rho_{g}(t) \ dt \simeq \frac{1}{N_T} \sum_{i=1}^{N_T} \rho_{g}(t_i) \quad \in [0,1]
\end{equation}

Finally if we denote with $\phi_{d_{k,k'}}(t):=\theta_k(t)-\theta_{k'}(t)$ the relative phase between two participants in the group at time $t$, it is possible to estimate their dyadic synchronization, that is the synchronization level between participants $k$ and $k'$ over the whole round:

\begin{equation}
\label{eqn:r4}
\rho_{d_{k,k'}} := \left | \frac{1}{T} \int_{0}^{T} e^{ \{ j \phi_{d_{k,k'}}(t) \} } \ dt \right |
\simeq \left | \frac{1}{N_T} \sum_{i=1}^{N_T} e^{ \{ j \phi_{d_{k,k'}}(t_i) \} } \right | \quad \in [0,1]
\end{equation}

It is worth pointing out that high dyadic synchronization levels can coexist with low group synchronization values.

\section{Testbed example}
\label{sec:testbed}
As a testbed scenario we consider the synchronization of rocking chairs motion studied in \cite{RGFGM12}. In particular, participants sit on six identical wooden rocking chairs disposed as a circle and are supposed to rock them in two different conditions:
\begin{enumerate}
\item \emph{Eyes closed:} participants are required to rock at their own preferred frequency while keeping their eyes closed;
\item \emph{Eyes open:} participants are required to rock at their own preferred frequency while trying to synchronize their rocking chair movements as a group.
\end{enumerate}

In the eyes closed condition the participants are not visually coupled, meaning that the oscillation frequency of each of them is not influenced by the motion of the others, whilst in the eyes open condition each player is asked to look at the middle of the circle in order to try and synchronize their motion with the one of the others. The six participants first perform a trial while keeping their eyes closed, then perform two eyes open trials, namely $T1$ and $T2$; each of the three trials lasts $3$ minutes. 

In Fig. \ref{fig:rockChSim} some results about the typical trend of the group synchronization $\rho_g(t)$ and its mean value and standard deviation are represented for each of the three aforementioned trials. In particular, in Fig. \ref{fig:rockChSim}a we can observe that the mean value $\rho_g$ of the group synchronization, represented as a circle, is around $0.4$ in the eyes closed condition, while it is around $0.85$ in the eyes open condition: this means that, when the participants are not visually coupled, as expected synchronization does not emerge, whilst when visually coupled and explicitly told to rock their own chair movements as a group, the synchronization level significantly increases. In Fig. \ref{fig:rockChSim}b we can observe that in the eyes closed condition the amplitude of the oscillations of the group synchronization is higher than the one obtained in the eyes open condition.

 \begin{figure}
 \centering
 \subfloat[mean value and standard deviation]{\includegraphics[width=.5\textwidth]{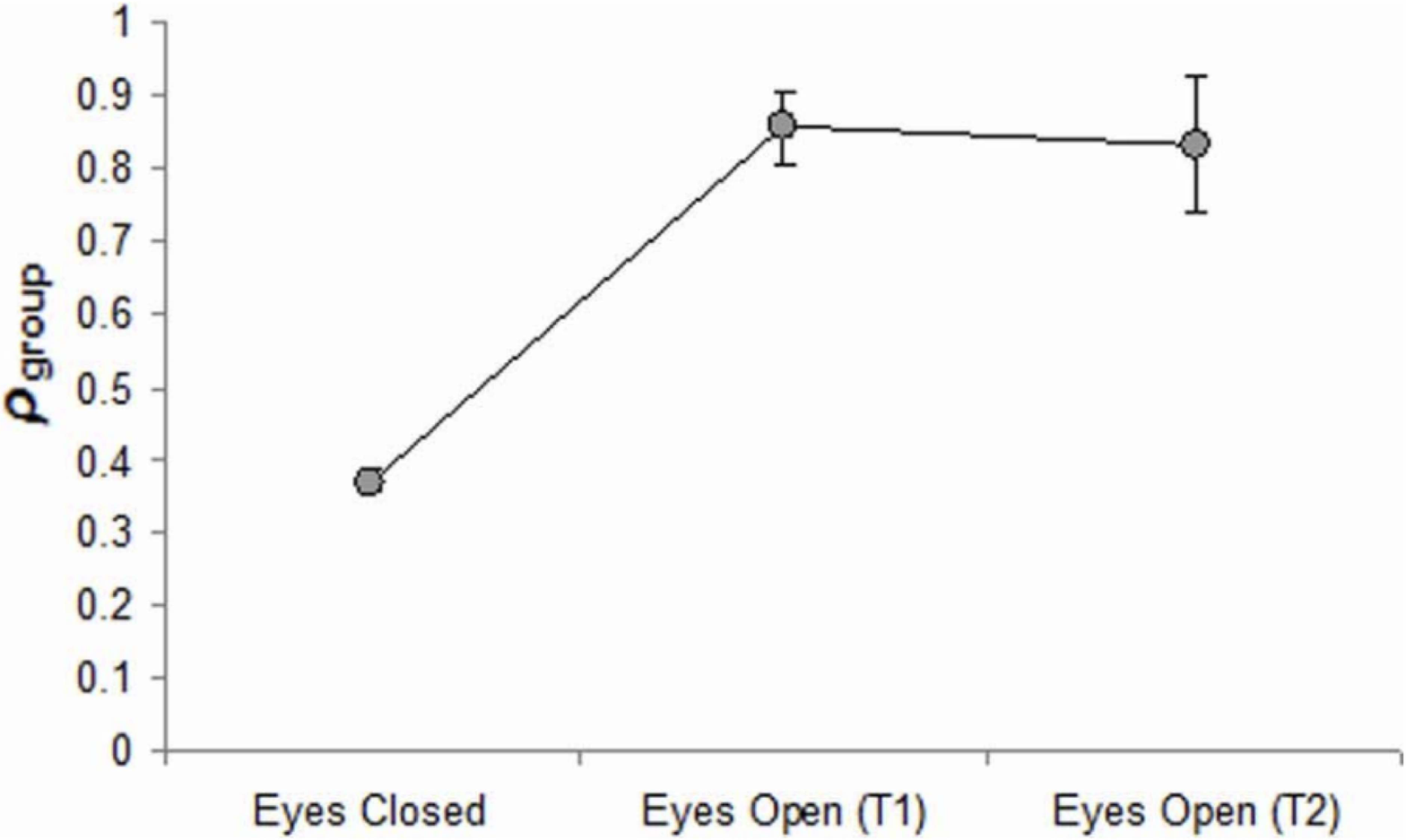}}
 \subfloat[typical trend]{\includegraphics[width=.5\textwidth]{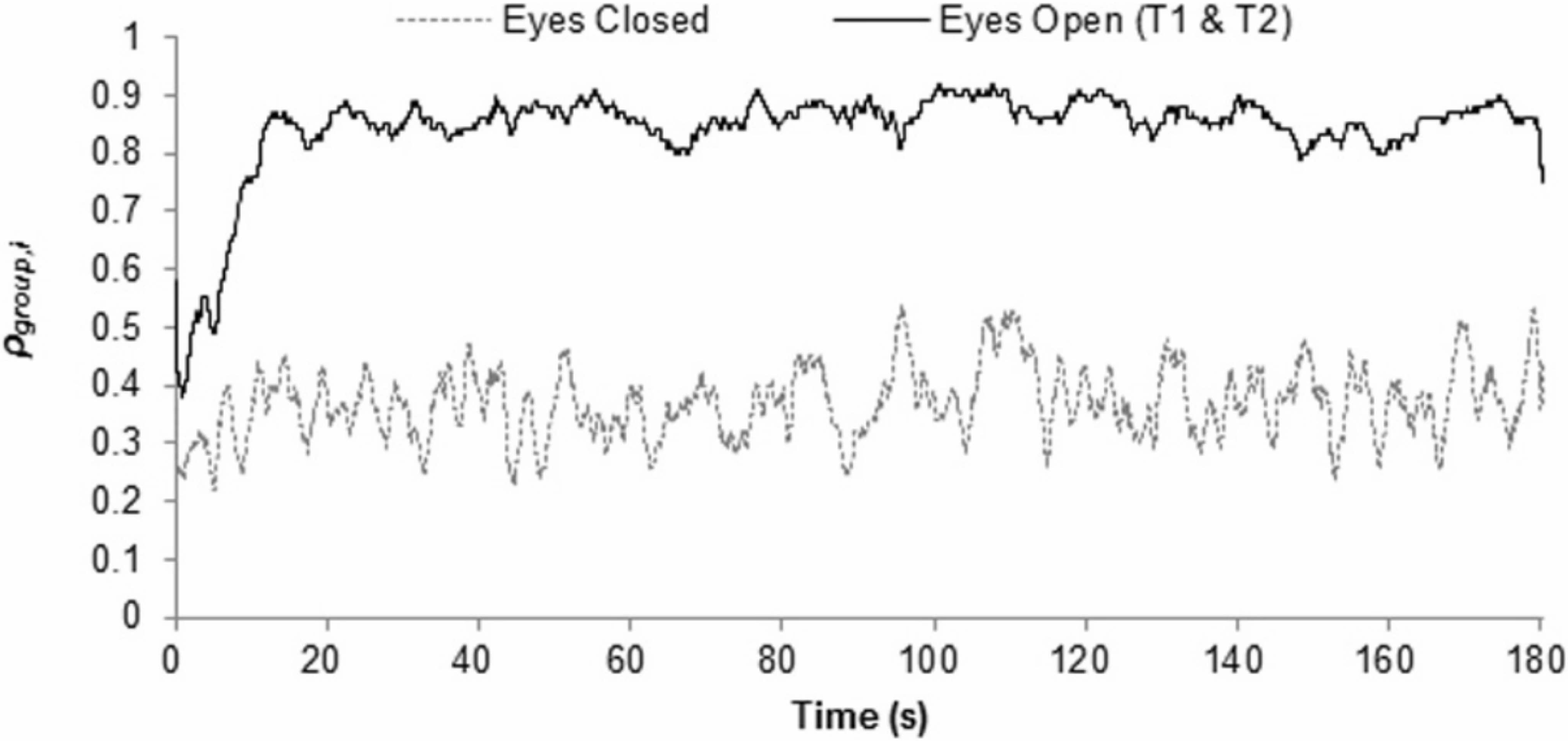}}
 \caption{Group synchronization in the rocking chairs experiments of \cite{RGFGM12} - T1 and T2 refer to two different trials of the eyes open condition}
 \label{fig:rockChSim}
\end{figure}

In Table \ref{table:rhoKcompRC} we show typical values of the degree of synchronization $\rho_k$ of the participants involved in the rocking chairs experiments, both for the eyes closed and the eyes open condition. It is easy to see that, as expected, the value of $\rho_k$ is much higher for almost all the participants when they are visually coupled. Interestingly enough, agent $6$ does not undergo an improvement of $\rho_6$ with respect to the eyes closed condition, meaning that such participant has more trouble synchronizing with the group compared to the other ones.

\begin{table}[ht]
\caption{Degree of synchronization of the participants in the rocking chairs experiments of \cite{RGFGM12} - EC: eyes closed, EO: eyes open}
\centering
\label{table:rhoKcompRC}
\begin{tabular}{lll}
\hline\noalign{\smallskip}
Participant & EC ($\rho_g=0.36$) & EO ($\rho_g=0.80$) \\
\noalign{\smallskip}\hline\noalign{\smallskip}
1 & $0.36$ & $0.95$ \\
2 & $0.34$ & $0.92$ \\
3 & $0.30$ & $0.95$ \\
4 & $0.35$ & $0.88$ \\
5 & $0.34$ & $0.67$ \\
6 & $0.40$ & $0.37$ \\
\noalign{\smallskip}\hline
\end{tabular}
\end{table}

In Table \ref{table:dyadSynchRC} we show typical values of the dyadic synchronization $\rho_{d_{k,k'}}$ of the participants involved in the rocking chairs experiments for the eyes open condition. As expected, the lowest values are the ones obtained with respect to the participant that most struggled to synchronize with the group, that is participant $6$.

\begin{table}[ht]
\caption{Dyadic synchronization of the participants in the rocking chairs experiments of \cite{RGFGM12} for the eyes open condition ($\rho_g=0.80$)}
\centering
\label{table:dyadSynchRC}
\begin{tabular}{llllll}
\hline\noalign{\smallskip}
Participants & $2$ & $3$  & $4$ & $5$ & $6$ \\ 
\noalign{\smallskip}\hline\noalign{\smallskip}
1 & $0.87$ & $0.86$  & $0.81$ & $0.63$ & $0.19$  \\
2 & $-$ & $0.85$ & $0.78$ & $0.59$ & $0.21$ \\
3 & $-$ & $-$  & $0.82$ & $0.61$ & $0.21$  \\
4 & $-$ & $-$  & $-$ & $0.50$ & $0.18$ \\
5 & $-$ & $-$  & $-$ & $-$ & $0.14$ \\
\noalign{\smallskip}\hline
\end{tabular}
\end{table} 

\subsection{Modeling results}
\label{sec:modres}
In this section we uncover the synchronization features that the three different interaction protocols introduced earlier lead to, with respect to the rocking chairs experiments introduced earlier as a testbed scenario \cite{RGFGM12}.
We will explore whether and how the model of coupled HKB oscillators we propose in this paper can reproduce the key features of the observed experimental results. In so doing we will explore:
\begin{itemize}
\item the effects of choosing different coupling functions;
\item how varying the coupling strength affects the synchronization level of the agents.
\end{itemize}

In what follows we simulate a heterogeneous network of $N=6$ HKB oscillators whose parameters and initial values are heuristically set as described in Table \ref{table:6nodesTable} and we set $T=200s$. We suppose that the network is simple, connected, unweighted and undirected and we assume that each node is connected to all the others (complete graph), which we believe well represents the topology implemented in the rocking chairs experiments of \cite{RGFGM12} for the eyes open condition.

\begin{table}[ht]
\caption{Numerical simulations - parameters and initial values for a network of $N=6$ HKB oscillators}
\centering
\label{table:6nodesTable}
\begin{tabular}{llllll}
\hline\noalign{\smallskip}
Nodes & $\alpha_i$ & $\beta_i$  & $\gamma_i$ & $\omega_i$ & $x_i(0)$ \\
\noalign{\smallskip}\hline\noalign{\smallskip}
1 & $0.46$ & $1.16$  & $0.58$ & $0.31$ & $[-1.4, +0.3]$  \\
2 & $0.37$ & $1.20$  & $1.84$ & $0.52$ & $[+1.0, +0.2]$ \\
3 & $0.34$ & $1.73$  & $0.62$ & $0.37$ & $[-1.8, -0.3]$  \\
4 & $0.17$ & $0.31$  & $1.86$ & $0.41$ & $[+0.2, -0.2]$ \\
5 & $0.76$ & $0.76$  & $1.40$ & $0.85$ & $[+1.5, +0.1]$ \\
6 & $0.25$ & $0.86$  & $0.56$ & $0.62$ & $[-0.8, -0.1]$ \\
\noalign{\smallskip}\hline
\end{tabular}
\end{table} 

\begin{figure}
 \centering
 \subfloat[mean value and standard deviation]{\includegraphics[width=.5\textwidth]{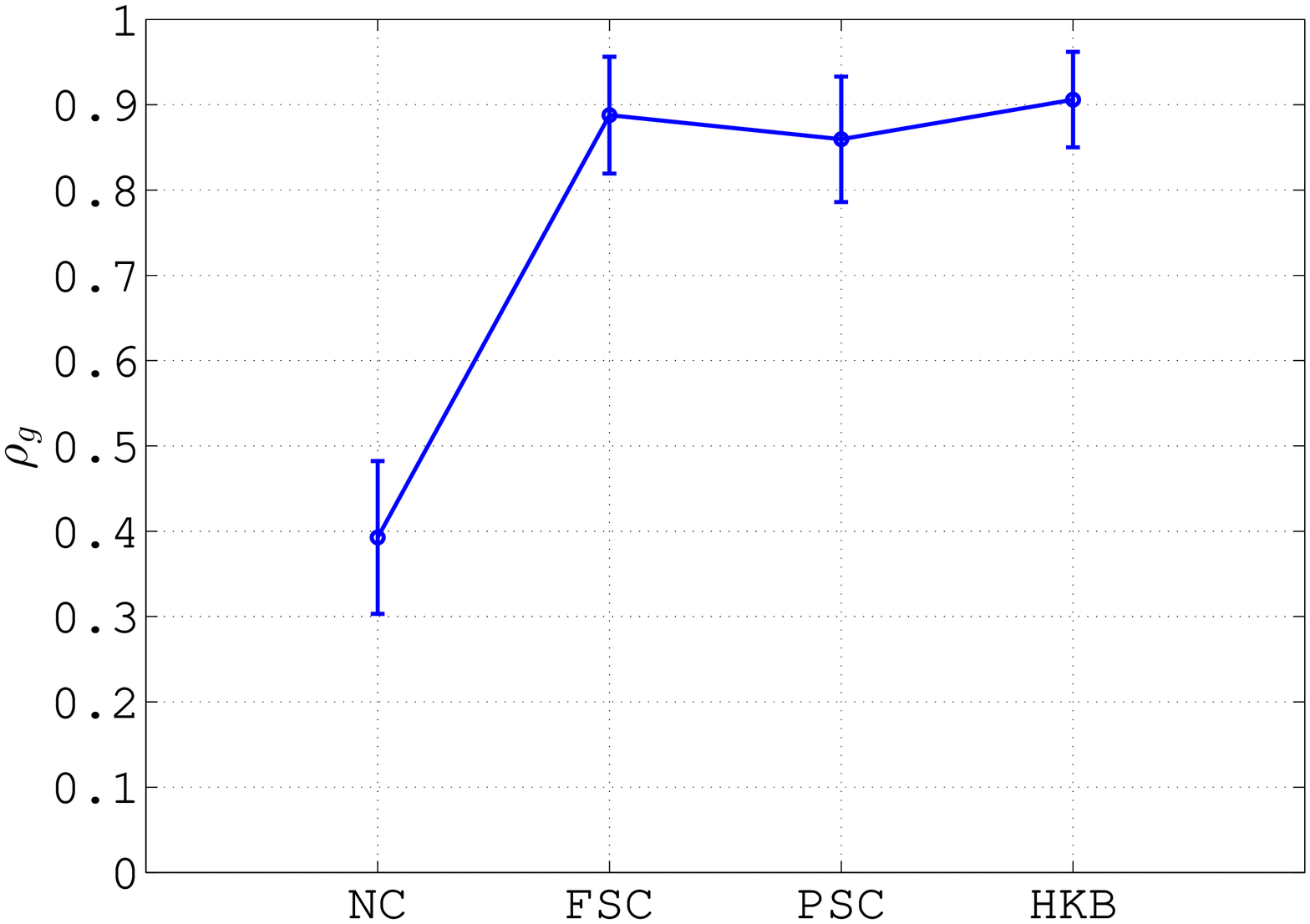}}
 \subfloat[typical trend]{\includegraphics[width=.5\textwidth]{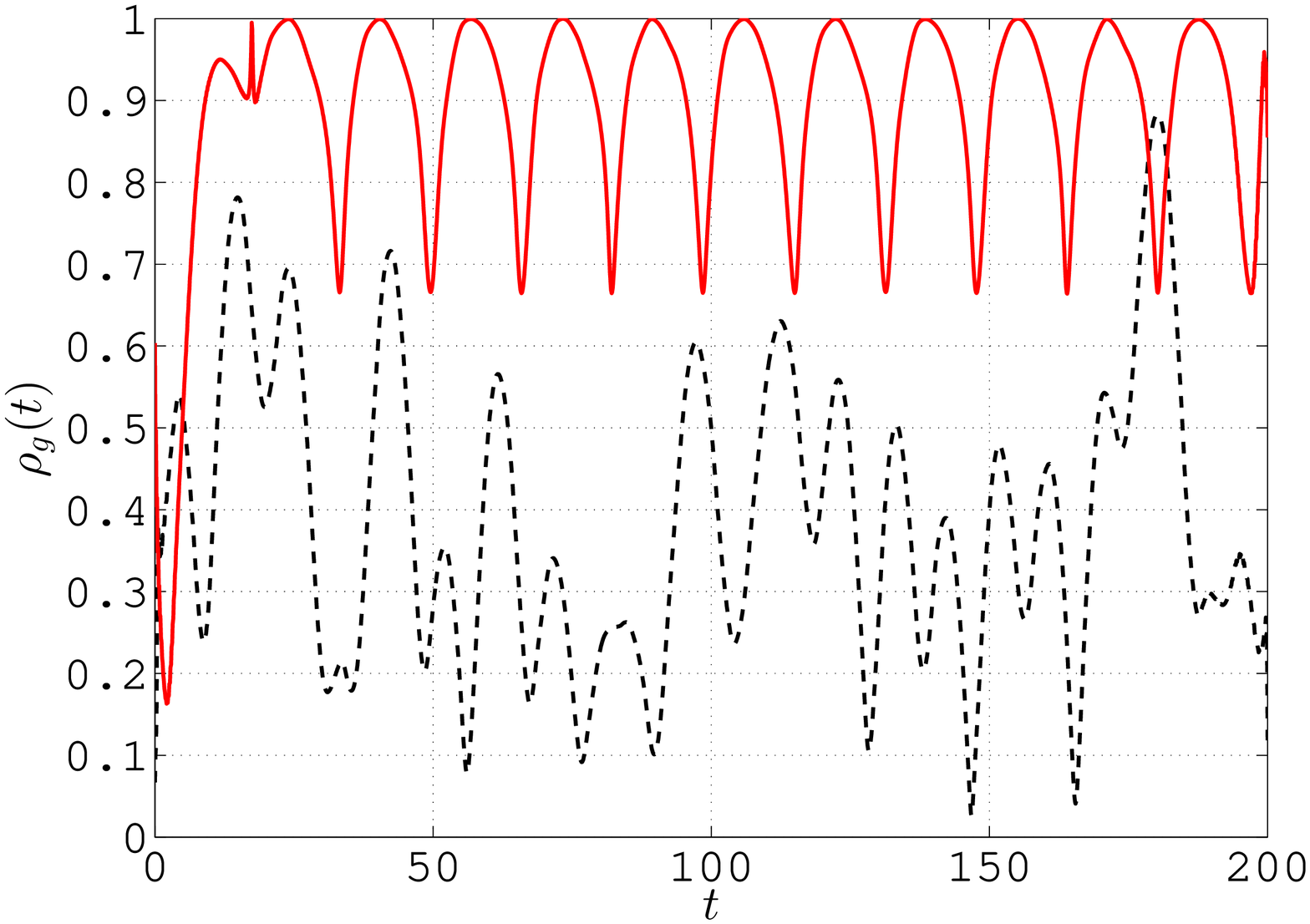}}
 \caption{Group synchronization in a heterogeneous unweighted complete graph of $N=6$ HKB oscillators - \textbf{a} \emph{NC}: no coupling, \emph{FSC}: full state coupling ($c=0.15$), \emph{PSC}: partial state coupling ($c_1=c_2=0.15$), \emph{HKB}: HKB coupling ($a=b=-1,c=0.15$) - \textbf{b} black dashed line: no coupling, red solid line: full state coupling}
 \label{fig:syncLev6}
\end{figure}

In particular, since we are interested in replicating the key features of the rocking chairs experiments for both the conditions (eyes open and eyes closed), in Fig. \ref{fig:syncLev6} we show the group synchronization obtained with and without interaction protocol. In particular, in Fig. \ref{fig:syncLev6}a we show mean value and standard deviation of $\rho_g(t)$: in each column, they are shown both for all the three interaction protocols presented earlier and for the case in which the nodes are not connected, respectively. The mean value is indicated with a circle, while the standard deviation from it is indicated with a vertical bar whose extremities are delimited by two horizontal lines.
In particular, if we denote with $\rho_g$ the mean value of the the group synchronization $\rho_g(t)$ obtained as defined in Eq. \ref{eqn:r3} for each of the four aforementioned cases after a simulation of duration $T$ and with $\sigma_{\rho_g}$ its standard deviation, respectively, we have that:

\begin{equation}
\sigma_{\rho_g} = \sqrt{ \frac{1}{T} \int_{0}^{T} \left( \rho_g(t)-\rho_g \right)^2 \ dt }
\simeq \sqrt{ \frac{1}{N_T} \sum_{i=1}^{N_T} \left( \rho_g(t_i)-\rho_g \right)^2 }
\end{equation}

It is easy to observe that, in absence of connections among the nodes, which corresponds to $u_i=0 \ \forall i \in [1,N]$, the group synchronization has a mean value approximately equal to $0.4$, while it significantly increases (approximately $0.9$) when connecting the nodes with any of the three interaction protocols introduced above. 
These results confirm the observations previously made for a network of six human people involved in rocking chairs experiments (see Fig. \ref{fig:rockChSim}a).
In particular, we have chosen $c=0.15$ for the full state interaction protocol, $c_1=c_2=0.15$ for the partial state interaction protocol, and $a=b=-1, c=0.15$ for the HKB interaction protocol. 

In Fig. \ref{fig:syncLev6}b we show the time evolution of the group synchronization $\rho_g(t)$ when the nodes are not connected at all (black dashed line) and when they are connected through a full state interaction protocol (red solid line): for the sake of clarity we do not show the trend of $\rho_g(t)$ obtained with a partial state and an HKB interaction protocol since they are qualitatively analogous to the one obtained with a full state interaction protocol. 
Our simulation results are able to reproduce another key feature observed in \cite{RGFGM12}: when the nodes are uncoupled, which corresponds to the eyes closed condition, the amplitude of the oscillations of $\rho_g(t)$ is higher than the one obtained when the nodes are coupled instead, which corresponds to the eyes open condition (see Fig. \ref{fig:rockChSim}b).

Then, in Table \ref{table:rhoKcomp} we show the degree of synchronization $\rho_k$ obtained for each node of the network, both in absence of coupling among the agents and in its presence.
It is easy to see that for each node $k$ in the network, $\rho_k$ has much higher values when any of the three interaction protocols is introduced, confirming what observed in \cite{RGFGM12} when asking the participants to rock their chairs while keeping their eyes open rather than closed. Moreover, we are able to reproduce another interesting feature: despite the group synchronization assuming high values when the human players are visually coupled, there might be some agents that struggle to keep up with the general trend of the group, therefore showing lower values in terms of $\rho_k$ (node $5$ in our simulations).

\begin{table}[ht]
\caption{Degree of synchronization of the nodes in a heterogeneous unweighted complete graph of $N=6$ HKB oscillators - no coupling (NC), full state coupling (FSC) with $c=0.15$, partial state coupling (PSC) with $c_1=c_2=0.15$ and HKB coupling (HKB) with $a=b=-1,c=0.15$}
\centering
\label{table:rhoKcomp}
\begin{tabular}{lllll}
\hline\noalign{\smallskip}
Node & NC & FSC & PSC & HKB \\
\noalign{\smallskip}\hline\noalign{\smallskip}
1 & $0.42$ & $0.95$ & $0.93$ & $0.97$ \\
2 & $0.38$ & $0.92$ & $0.94$ & $0.96$ \\
3 & $0.45$ & $0.98$ & $0.95$ & $0.97$ \\
4 & $0.41$ & $0.98$ & $0.96$ & $0.98$ \\
5 & $0.33$ & $0.33$ & $0.33$ & $0.49$ \\
6 & $0.36$ & $0.98$ & $0.97$ & $0.98$ \\
\noalign{\smallskip}\hline
\end{tabular}
\end{table} 

Furthermore, in Table \ref{table:dyadSynch} we show the dyadic synchronization $\rho_{d_{k,k'}}$ for all the possible couples of nodes in the network: again, our simulation results confirm what observed for the rocking chairs experiments. Indeed, the couples of nodes with lower dyadic synchronization correspond to pairs in which at least either of the two nodes had trouble synchronizing with the general trend of the group (node $5$ in our simulations). For the sake of clarity we show $\rho_{d_{k,k'}}$ only when connecting the nodes through a full state interaction protocol, since analogous results can be obtained also for the other two strategies introduced earlier in this paper.

\begin{table}[ht]
\caption{Dyadic synchronization in a heterogeneous unweighted complete graph of $N=6$ HKB oscillators - full state coupling ($c=0.15$) }
\centering
\label{table:dyadSynch}
\begin{tabular}{llllll}
\hline\noalign{\smallskip}
Nodes & $2$ & $3$  & $4$ & $5$ & $6$ \\
\noalign{\smallskip}\hline\noalign{\smallskip}
1 & $0.91$ & $0.98$  & $0.94$ & $0.38$ & $0.98$  \\
2 & $-$ & $0.92$ & $0.96$ & $0.43$ & $0.93$ \\
3 & $-$ & $-$  & $0.97$ & $0.39$ & $0.99$  \\
4 & $-$ & $-$  & $-$ & $0.40$ & $0.98$ \\
5 & $-$ & $-$  & $-$ & $-$ & $0.41$ \\
\noalign{\smallskip}\hline
\end{tabular}
\end{table} 

It is easy to foresee that, regardless of the interaction protocol the nodes are connected through, the value of the coupling strength has a direct impact on the group synchronization in terms of its mean value and its standard deviation. We now show how quantitatively $\rho_g$ varies as the the coupling strength varies for all the three interaction protocols introduced earlier in this paper, when considering a heterogeneous unweighted complete graph of $N=6$ HKB oscillators whose parameters and initial values are described in Table \ref{table:6nodesTable}. Moreover, we set $T=200s$.

From Fig. \ref{fig.rhoG_coupStr_FSC} to \ref{fig:rhoG_coupStr_HKB} we show mean value and standard deviation of the group synchronization $\rho_g(t)$ obtained for different values of the coupling strength when considering full state coupling, partial state coupling and HKB coupling as interaction protocols, respectively. In particular, the blue solid line refers to the mean value of $\rho_g(t)$, while the red dashed lines indicate the variation or dispersion from it.

\begin{figure}
 \centering
 \subfloat[group synchronization]{\includegraphics[width=.5\textwidth]{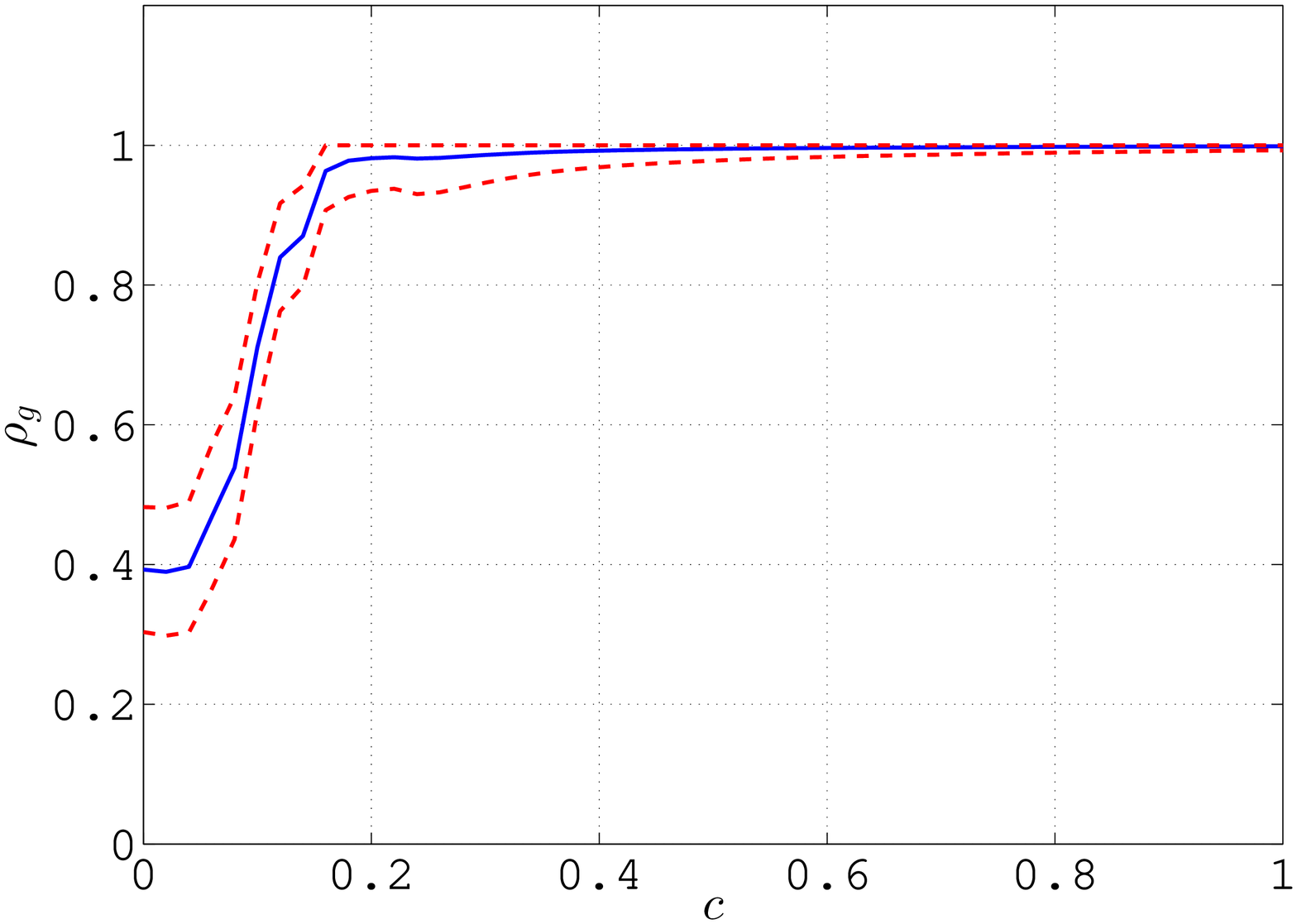}}
 \subfloat[group synchronization - zoom]{\includegraphics[width=.5\textwidth]{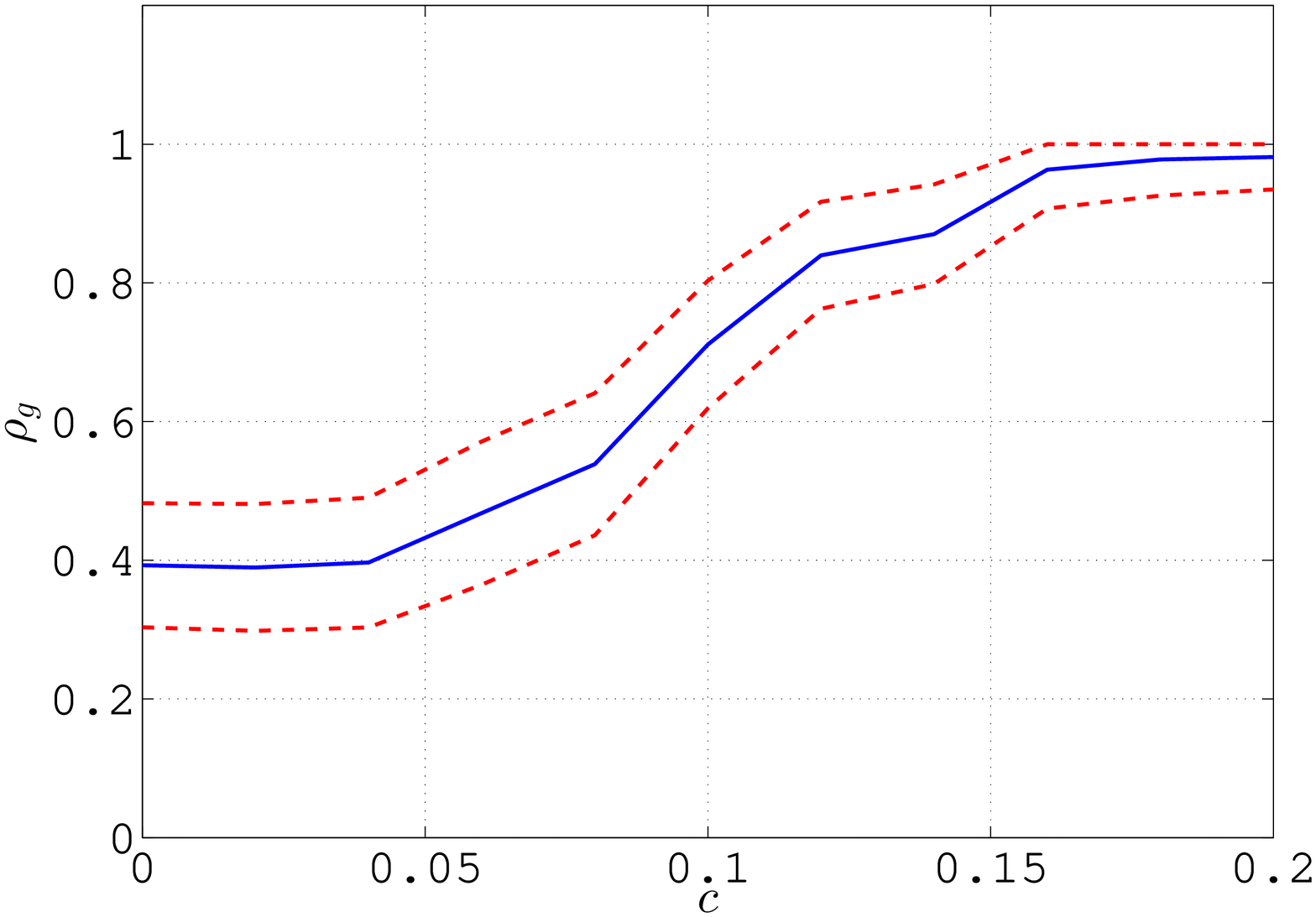}}
 \caption{Mean value and standard deviation of the group synchronization in a heterogeneous unweighted complete graph of $N=6$ HKB oscillators for different values of the coupling strength $c$ - full state coupling}
 \label{fig.rhoG_coupStr_FSC}
\end{figure}

From Fig. \ref{fig.rhoG_coupStr_FSC}a it is clear that, when considering a full state coupling as interaction protocol, the group synchronization increases as the coupling strength $c$ increases: in particular, in order for the network to well synchronize it is sufficient a relatively small value for the coupling strength ($c \simeq 0.15$, see Fig. \ref{fig.rhoG_coupStr_FSC}b). In terms of multiplayer games, this means that the stronger the influence that each player has on the others, the better the overall synchronization of the human participants.

\begin{figure}
 \centering
 \subfloat[$c_2=0, c_1$ variable]{\includegraphics[width=.5\textwidth]{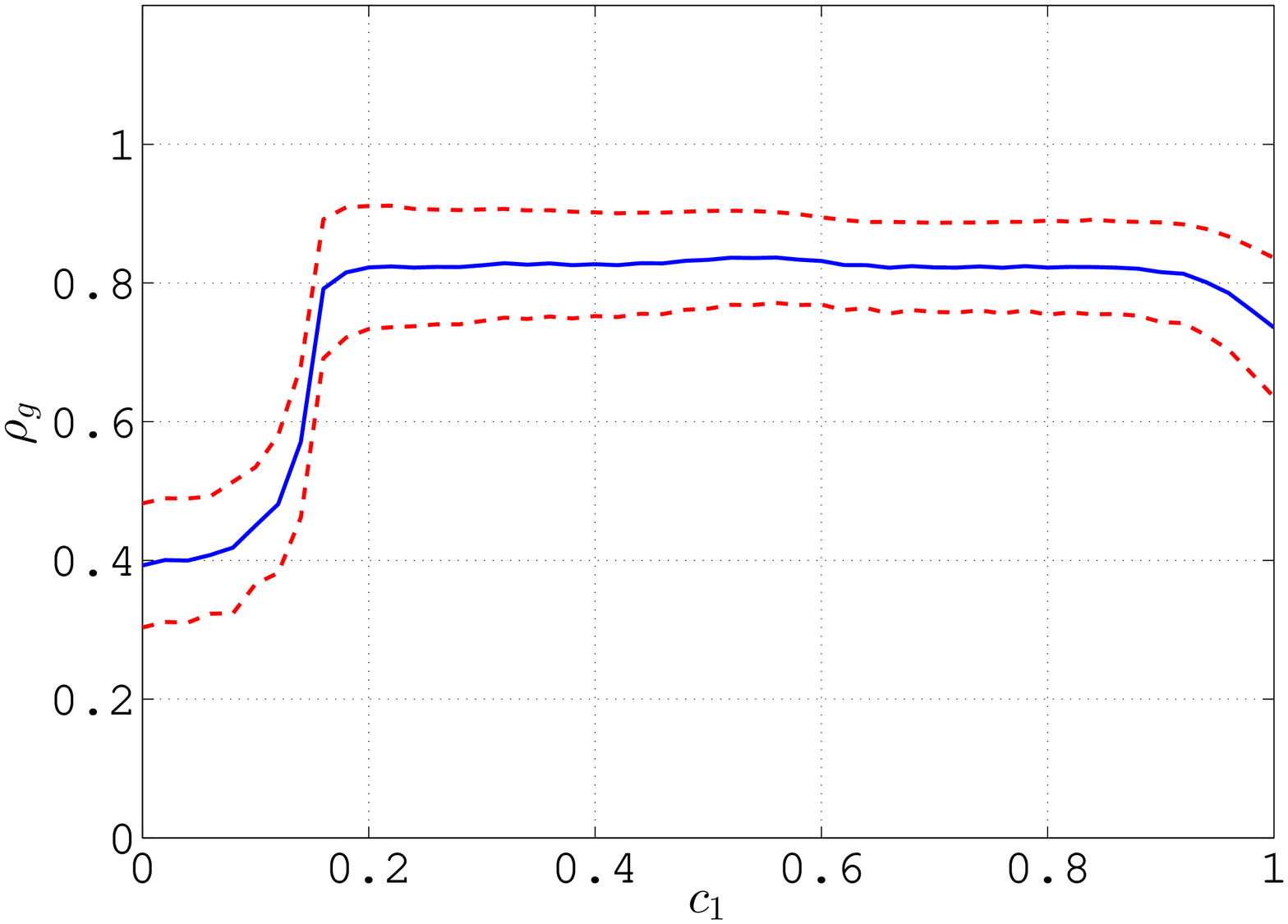}}
 \subfloat[$c_1=0, c_2$ variable]{\includegraphics[width=.5\textwidth]{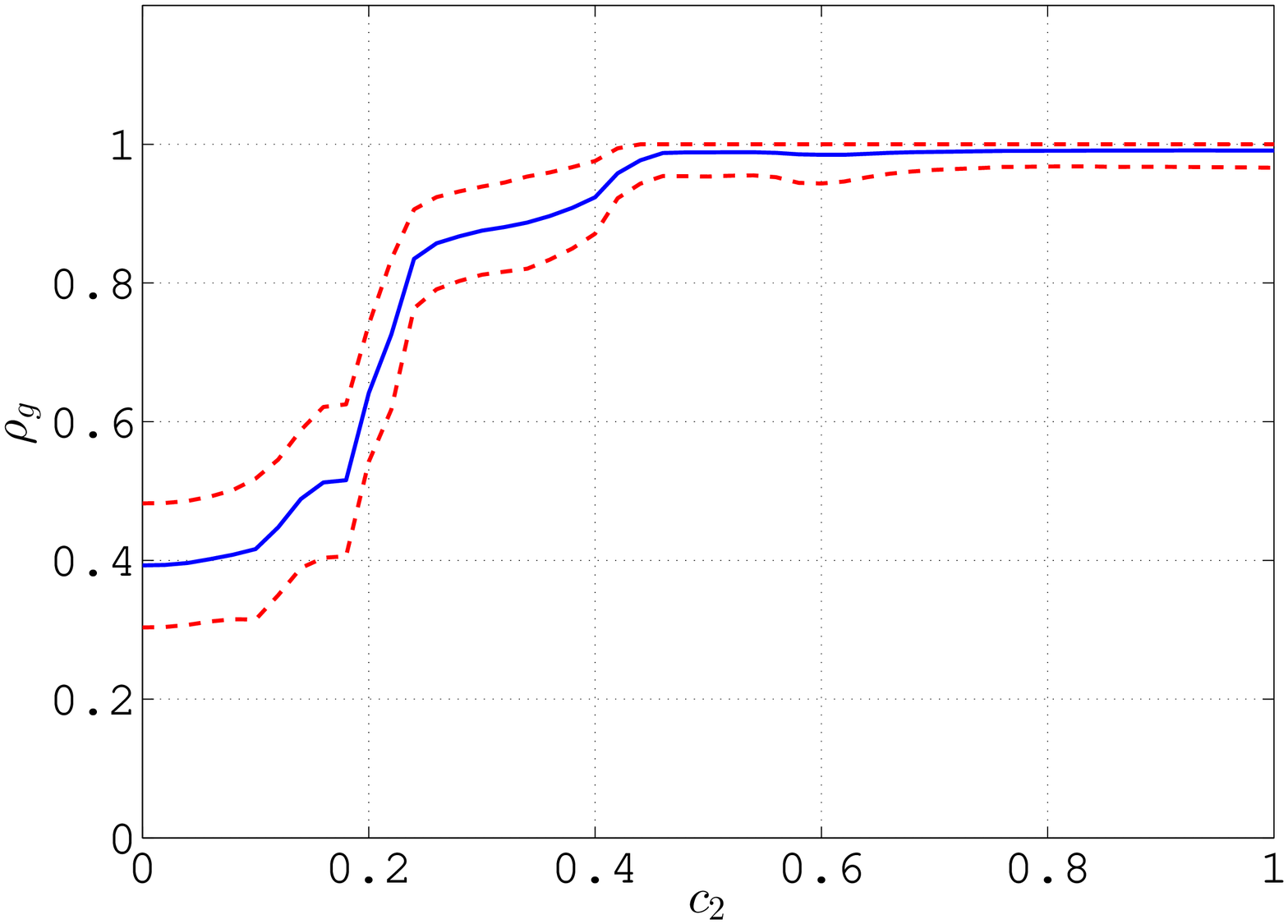}}
 \caption{Mean value and standard deviation of the group synchronization in a heterogeneous unweighted complete graph of $N=6$ HKB oscillators for different values of the coupling strengths $c_1$ and $c_2$ - partial state coupling}
 \label{fig:rhoG_coupStr_PSC}
\end{figure}

In Fig. \ref{fig:rhoG_coupStr_PSC}a we show how, when considering a partial state coupling as interaction protocol, the group synchronization varies for increasing values of $c_1$ while keeping $c_2$ constantly equal to $0$, and vice versa in Fig. \ref{fig:rhoG_coupStr_PSC}b. As we can see, the influence that $c_2$ has on the group synchronization is stronger than the one provided by $c_1$. In terms of multiplayer games, this means that human players react better to changes in the velocity of their neighbors rather than in their position. This results is confirmed also in Fig. \ref{fig:rhoG_coupStr_PSC_HKB}a in which we show how the mean value of the group synchronization changes as $c_1$ and $c_2$ are simultaneously varied (darker colors refer to lower values of the average group synchronization, whilst lighter ones refer to higher values).

\begin{figure}
 \centering
 \subfloat[group synchronization]{\includegraphics[width=.5\textwidth]{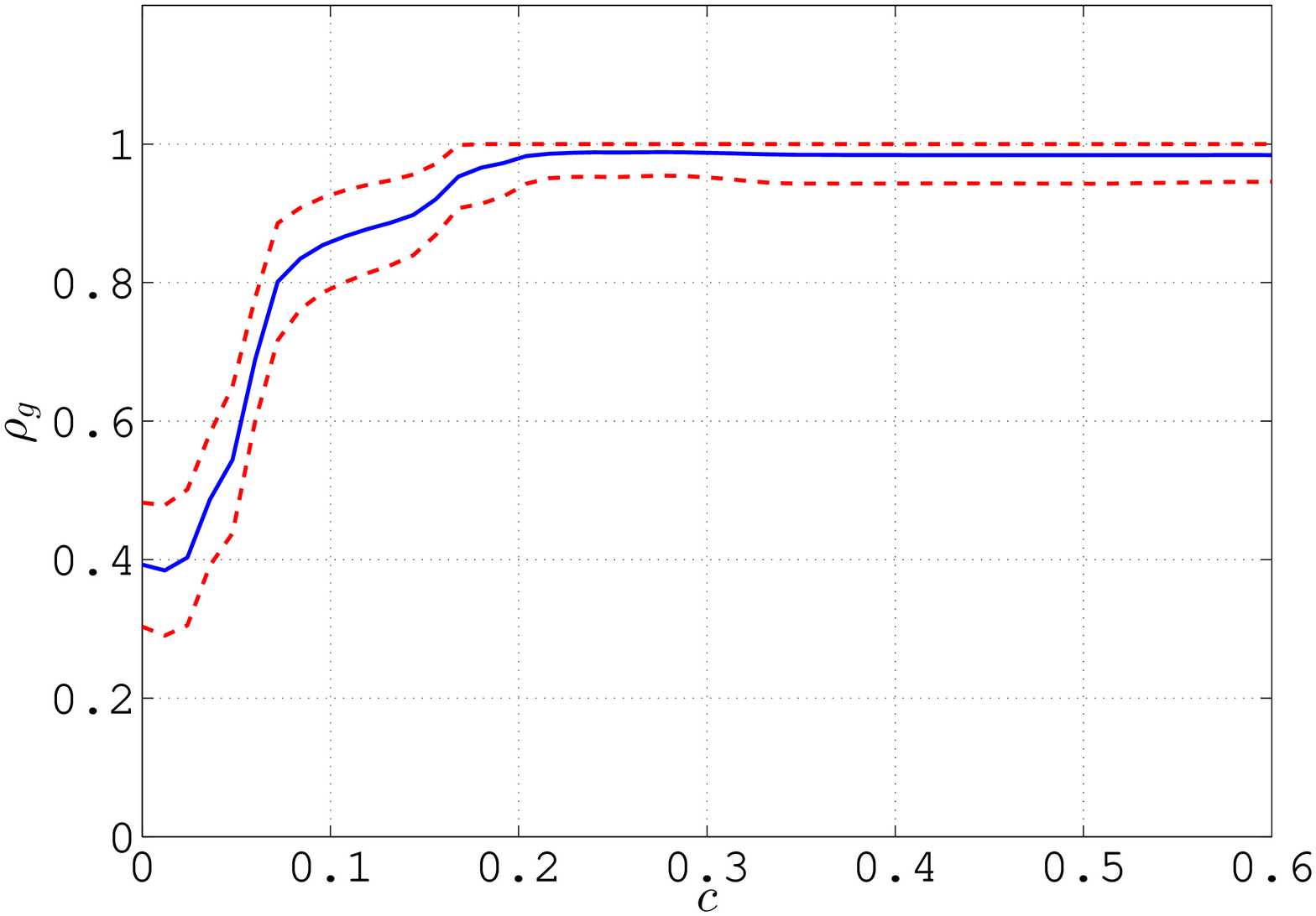}}
 \subfloat[group synchronization - zoom]{\includegraphics[width=.5\textwidth]{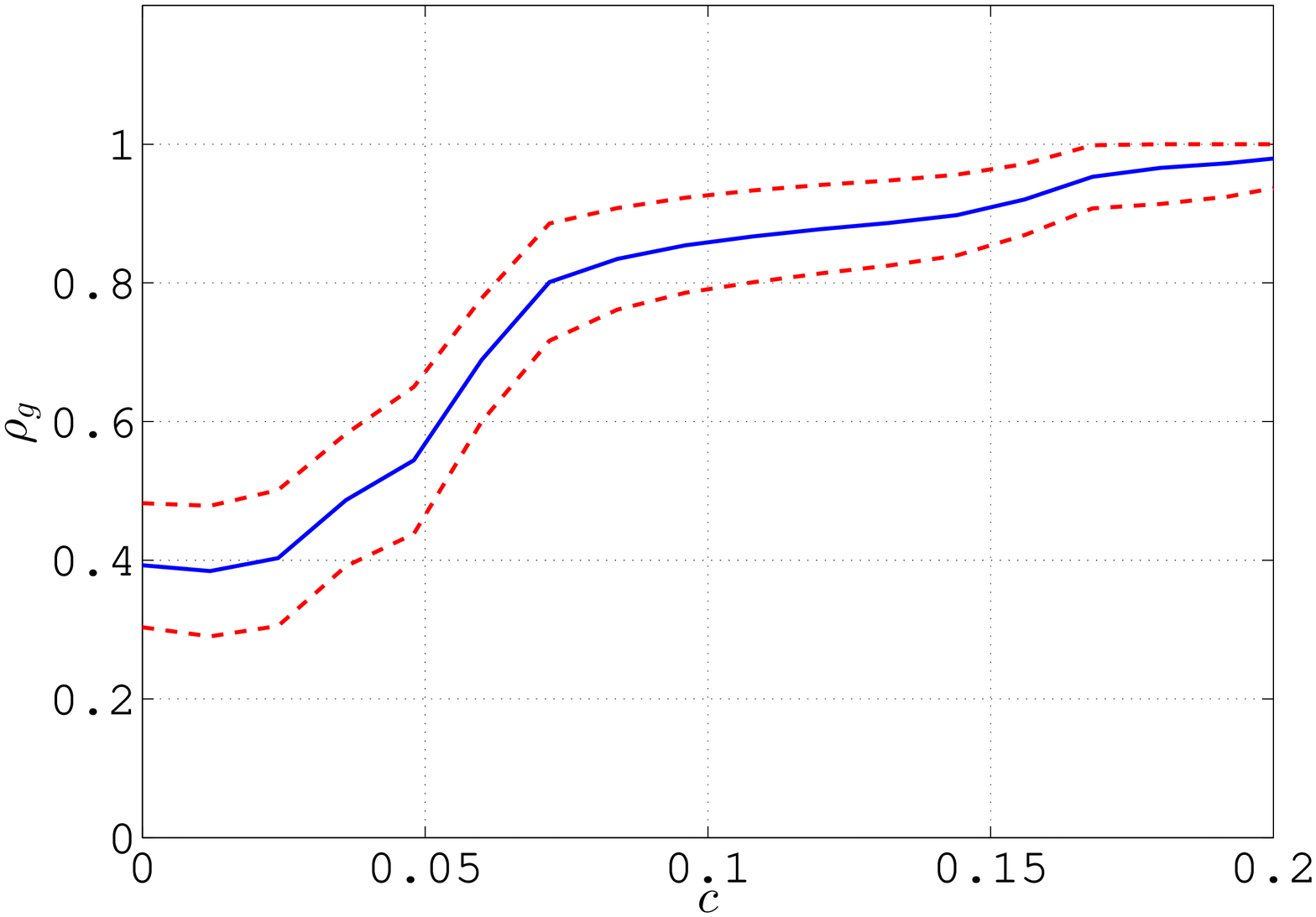}}
 \caption{Mean value and standard deviation of the group synchronization in a heterogeneous unweighted complete graph of $N=6$ HKB oscillators for different values of the coupling strength $c$ while keeping $a=b=-1$ constant - HKB coupling}
 \label{fig:rhoG_coupStr_HKB}
\end{figure}

Finally from Fig. \ref{fig:rhoG_coupStr_HKB}a it is clear that, when considering an HKB coupling as interaction protocol while keeping $a$ and $b$ constantly equal to $-1$, the group synchronization increases as the coupling strength $c$ increases. In particular, like in the case of a full state interaction protocol, in order for the network to well synchronize it is sufficient to choose a relatively small value for the coupling strength ($c \simeq 0.15$, see Fig. \ref{fig:rhoG_coupStr_HKB}b). In terms of multiplayer games, this means that the stronger the influence that each player has on the others, the better the overall synchronization of the human participants. This results is confirmed also in Fig. \ref{fig:rhoG_coupStr_PSC_HKB}b in which we show how the mean value of the group synchronization changes as $a$ and $b$ are simultaneously varied while keeping $c$ constantly equal to $1$ (darker colors refer to lower values of the average group synchronization, whilst lighter ones refer to higher values). As we can see, as the values of $|a|$ and $|b|$ increase, then the average of the group synchronization increases as well.

\begin{figure}
 \centering
 \subfloat[partial state coupling - $c_1,c_2$ variable]{\includegraphics[width=.5\textwidth]{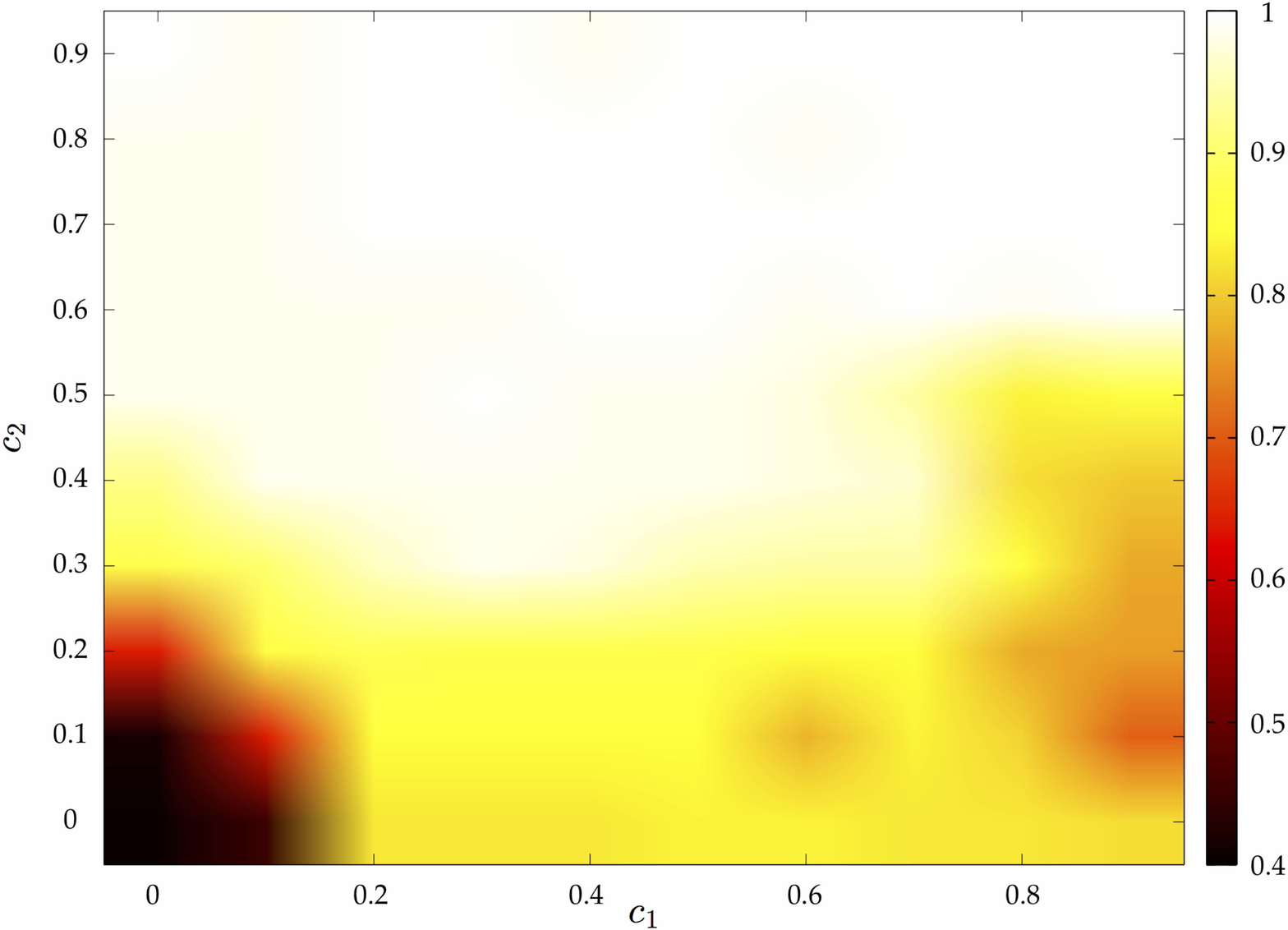}}
 \subfloat[HKB coupling - $a,b$ variable while keeping $c=1$ constant]{\includegraphics[width=.5\textwidth]{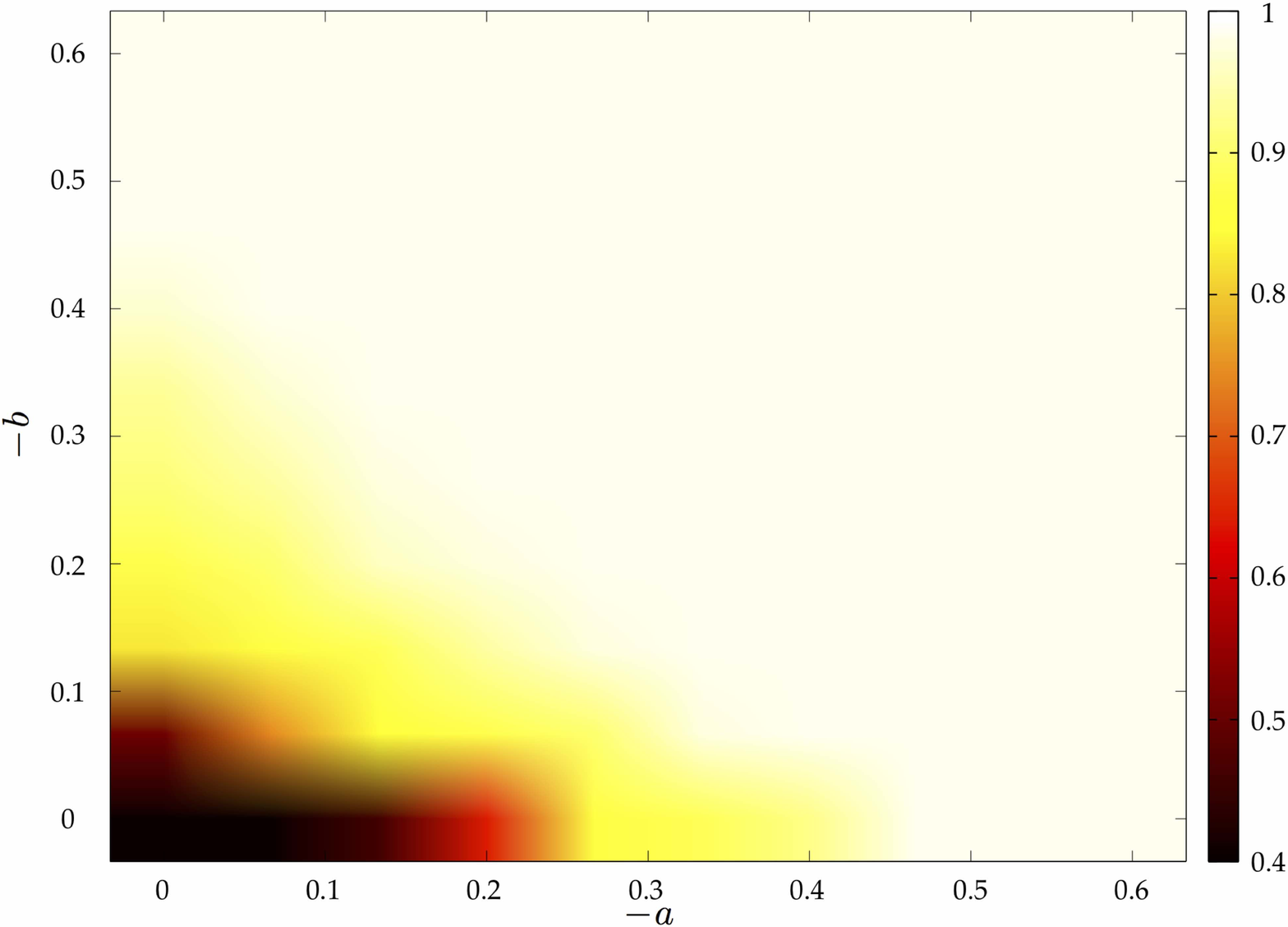}}
 \caption{Mean value of the group synchronization $\rho_g(t)$ in a heterogeneous unweighted complete graph of $N=6$ HKB oscillators for different values of the coupling strengths}
 \label{fig:rhoG_coupStr_PSC_HKB}
\end{figure}

\section{Entrainment of the network}
\label{sec:entrainment}
In this section we analyze the effects on the group synchronization of the network of adding an external sinusoidal signal to the dynamics of each node. Our main objective is understanding whether, and possibly under what conditions, such entrainment signal leads to a better synchronization level of a heterogeneous network of HKB oscillators with respect to the case in which the signal is absent. This will help us understand whether an external auditory or visual stimulus can improve the coordination level in multiplayer games when considering networks of human people involved in some synchronization task. 

Following the approach of \cite{RdBS10}, we model such a scenario in the following way:

\begin{equation}
\label{eqn:hkbExtSig}
f_i(t,x_i)=
\begin{bmatrix}
x_{i_2} \\
- (\alpha_i x_{i_1}^2+\beta_i x_{i_2}^2-\gamma_i)x_{i_2} - \omega_i^2 x_{i_1} + \zeta
\end{bmatrix}+u_i
\end{equation}
where $\zeta(t)=A_\zeta \sin \left(\omega_\zeta t \right)$ represents the entrainment signal and $u_i(t)$ one of the interaction protocols introduced earlier in this paper.

We introduce the \emph{entrainment index} $\rho_E \in [0,1]$ in order to quantify the overall synchronization level between the network and the external signal $\zeta(t)$:

\begin{equation}
\label{eqn:entrIndexHKB}
\rho_{E_k} := \left | \frac{1}{T} \int_{0}^{T} e^{  j [ \theta_k(t)- \theta_\zeta (t) ]  } \ dt \right |, \ \rho_{E} := \frac{1}{N} \sum_{k=1}^{N} \rho_{E_k}
\end{equation}
where $\theta_k(t)$ is the phase of the $k$-th node, $\theta_\zeta(t)$ is the phase of $\zeta(t)$, $T$ is the duration of the experiment and $N$ is the number of nodes in the network. The closer $\rho_E$ is to $1$, the better the synchronization of the group with the entrainment signal.

In what follows we simulate a heterogeneous network of $N=6$ HKB oscillators whose parameters and initial values are heuristically set as described in Table \ref{table:6nodesTable} and we set $T=200s$. We suppose that the network is simple, connected, unweighted and undirected and we assume that each node is connected to all the others (complete graph), which we believe well represents the topology implemented in the rocking chairs experiments of \cite{RGFGM12}.

\begin{figure}
 \centering
 \includegraphics[width=.7\textwidth]{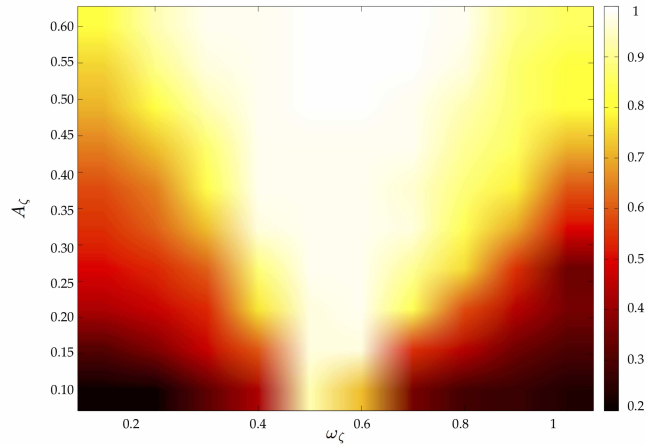}
 \caption{Entrainment index in a heterogeneous unweighted complete graph of $N=6$ HKB oscillators - full state coupling ($c=0.15$)}
 \label{fig:rhoEplotFSC}
\end{figure}

In Fig. \ref{fig:rhoEplotFSC} we show the entrainment index for different values of the frequency $\omega_\zeta$ and the amplitude $A_\zeta$ of the entrainment signal $\zeta(t)$ when considering a full state coupling as interaction protocol with $c=0.15$ (darker colors refer to lower values of $\rho_E$, whilst lighter ones refer to higher values). It is easy to see that, for each value of $\omega_\zeta$, the entrainment index increases as $A_\zeta$ increases as well, meaning that the network better synchronizes with $\zeta(t)$ for increasing values of its amplitude. Moreover, for a given value of $A_\zeta$, the highest values of $\rho_E$ are achieved when the frequency of the entrainment signal is close to the average value $\Omega$ of the natural frequencies $\omega_i$ of the nodes (in this case $\Omega \simeq 0.5$ ). These results confirm the findings of \cite{SRAG07,VSR15}, in which it is shown that spontaneous unintentional synchronization between the oscillation of a handheld pendulum swung by an individual and of an external sinusoidal stimulus (which corresponds to our external entrainment signal) emerges only when the frequency of the signal itself is similar to the preferred frequency of the player. For the sake of brevity we do not show how $\rho_E$ varies as $\omega_\zeta$ and $A_\zeta$ vary as well when considering partial state coupling and HKB coupling as interaction protocols, since we obtain results which are analogous to the ones shown in Fig. \ref{fig:rhoEplotFSC} for a full state coupling.

\begin{figure}
 \centering
 \includegraphics[width=.7\textwidth]{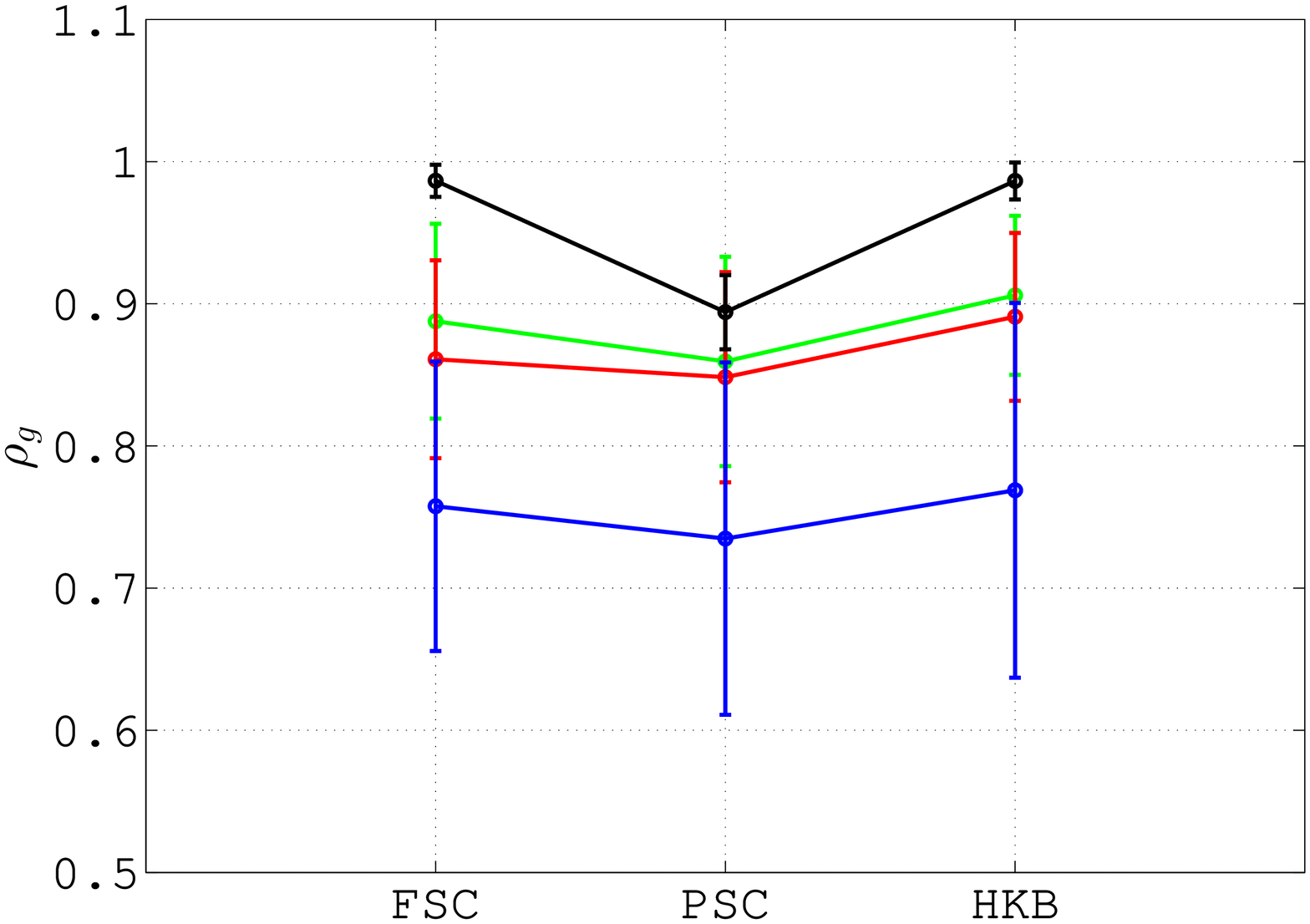}
 \caption{Mean value and standard deviation of the group synchronization in a heterogeneous unweighted complete graph of $N=6$ HKB oscillators - \emph{FSC}: full state coupling ($c=0.15$) - \emph{PSC}: partial state coupling ($c_1=c_2=0.15$) - \emph{HKB}: HKB coupling ($a=b=-1, c=0.15$) - green line: no entrainment signal, red line: $\omega_\zeta=0.1, A_\zeta=0.1$, blue line: $\omega_\zeta=0.3, A_\zeta=0.2$, black line: $\omega_\zeta=0.5, A_\zeta=0.3$}
 \label{fig:rhoGzetaComp}
\end{figure}

In Fig. \ref{fig:rhoGzetaComp} we show mean value and standard deviation of the group synchronization $\rho_g(t)$ when considering different parameters of the entrainment signal (green line: no entrainment signal, red line: $\omega_\zeta=0.1, A_\zeta=0.1$, blue line: $\omega_\zeta=0.3, A_\zeta=0.2$, black line: $\omega_\zeta=0.5, A_\zeta=0.3$) for all the three coupling protocols we have presented (\emph{FSC}: full state coupling ($c=0.15$) - \emph{PSC}: partial state coupling ($c_1=c_2=0.15$) - \emph{HKB}: HKB coupling ($a=b=-1, c=0.15$)). Since we are interested in understanding whether an additive external sinusoidal signal can improve the synchronization level of the network with respect to the case in which it is absent, the values of the coupling strengths chosen in these simulations for all the three interaction protocols are the same as the ones previously used in absence of entrainment signal (see Fig. \ref{fig:syncLev6}a). From Fig. \ref{fig:rhoGzetaComp} it is easy to observe that, for all the three interaction protocols, the group synchronization of the network improves only when the entrainment index $\rho_E$ has high values (see black line compared to the green one). In the other two cases (blue line and red line), the entrainment signal acts as a disturbance for the dynamics of the nodes and the group synchronization decreases. In terms of multiplayer games for networks of human people, this means that it is possible to further enhance the coordination level of participants only when the entrainment signal has an oscillation frequency which is close to the average of the natural oscillation frequencies of the individuals involved and its amplitude is sufficiently high.

\section{Convergence analysis}
\label{sec:mainresults}
As anticipated earlier, other than finding a mathematical model able to reproduce features observed experimentally in some multiplayer games studied in the existing literature, we are also interested in understanding under what conditions synchronization is observed to emerge. In particular, in this section we are going to show that global bounded synchronization can be analytically guaranteed for a heterogeneous network of diffusively coupled $N$ HKB oscillators by making use of two different approaches, namely \emph{contraction theory} and \emph{Lyapunov theory}.

\subsection{Contraction theory}
Let $|\cdot|$ be a norm defined on a vector $w \in \mathbb{R}^n$ with induced matrix norm $||\cdot||$. As stated in \cite{RdBS13}, given a matrix $P \in \mathbb{R}^{n \times n}$, the \emph{induced matrix measure} is defined as $\mu(P) := \lim_{h \to 0^+} \frac{\left( ||I+hP|| -1 \right)}{h}$.
\begin{dfn}
Let us consider the system $\dot{w} = F(t,w)$ defined $\forall t \ge 0, \ \forall w \in C \subset \mathbb{R}^n$. We say that such system is contracting with respect to a norm $|\cdot|$ with associated matrix measure $\mu (\cdot)$ iff
\begin{equation}
\label{eqn:ContrDef}
\exists \ k>0: \mu \left( J(w,t) \right) \le -k, \quad \forall w \in C, \forall t \ge 0
\end{equation}
where $J$ is the Jacobian of the system.
\end{dfn}

They key stage in the application of contraction theory to synchronization of networks of oscillators is the construction of the so-called \emph{virtual system} \cite{JS04}.

\begin{dfn}
\label{dfn:vsdfn}
Let us consider a heterogenous network described by Eq. \ref{eqn:networkeq}. The virtual system is defined as the system that has the trajectories of the nodes as particular solutions.
\end{dfn}

Formally, the virtual system depends on the state variables of the oscillators in the network and on some virtual state variables. Substitution of the state variables of a certain node $i$ into the virtual ones returns the dynamics of the $i$-th node of the network itself.

It is worth pointing out that virtual systems are originally defined for networks of identical systems in order to prove complete synchronization: indeed, it is possible to prove that if a virtual system is contracting, then $\lim_{t \to \infty} \eta(t) = 0$.
However, it is possible to define virtual systems also for networks of nonidentical oscillators by averaging the values of all the different parameters in order to prove bounded synchronization (see heterogeneous network of repressilators in \cite{RdB09} ).

In \cite{RdB09} a simple algorithm is provided that allows to check whether the virtual system of a certain heterogeneous network of $N$ agents is contracting, which leads to global bounded synchronization of the network itself. In particular, rather than verifying Eq. \ref{eqn:ContrDef} in order to check whether the virtual system is contracting, the algorithm consists in checking the truth of some statements regarding the single elements of the Jacobian of the virtual system and imposing some conditions:

\begin{enumerate}
\item build the Jacobian $J$ of the virtual system;
\item check whether the following statements are true or false
\begin{itemize}
\item S1: $J(i,i)<0$;
\item S2: $J(i,i)=-\rho_i, \ 0<\rho_i<\infty$;
\item S3: $J(i,j)\neq 0 \Rightarrow J(j,i)=0$;
\end{itemize}
\item generate a set of conditions for synchronization (CFS) according to the truth or the falsity of the previous statements.
\end{enumerate}

In particular, denoting with $n_{0_i}$ the number of $0$ elements in the $i-th$ row of the Jacobian of the virtual system, the CFS are generated in the following way:

\begin{itemize}
\item $S1, S2, S3 \Rightarrow |J(i,j)|<\frac{\rho_i}{n-n_{0_i}-1}$ ;
\item $S1, S2, \bar{S3} \Rightarrow |J(i,j)|>\frac{\rho_i}{n-n_{0_i}-1}, |J(j,i)|<\frac{\rho_j}{n-n_{0_j}-1}$ or vice versa ;
\item $S1, \bar{S2}, S3 \Rightarrow |J(i,j)|<\frac{|J(i,i)|}{n-n_{0_i}-1}$ ;
\item $S1, \bar{S2}, \bar{S3} \Rightarrow |J(i,j)|>\frac{|J(i,i)|}{n-n_{0_i}-1}, |J(j,i)|<\frac{|J(j,j)|}{n-n_{0_j}-1}$ or vice versa .
\end{itemize}

Note that if statement \emph{S1} is not true, it is not possible for the virtual system to be contracting.

\begin{thm}
\label{thm:ctheorythm}
Suppose to have a heterogeneous network of $N$ HKB oscillators interconnected via a full state coupling as described in Eq. \ref{eqn:gsldc}. Let us also assume that the network topology is a connected, simple, undirected and unweighted complete graph. If the following hypothesis is satisfied

\begin{equation}
\label{eqn:coupStrenCT2}
\frac{N-1}{N} \left( 2\tilde{\alpha}z_{1_{max}}z_{2_{max}}+\tilde{\omega}^2+\tilde{\gamma} \right) < c < \frac{N-1}{N}
\end{equation}
where $\tilde{\alpha}, \tilde{\omega}, \tilde{\gamma}$ are the average values of parameters $\alpha_i$, $\omega_i$, $\gamma_i$, respectively, and $z_{1_{max}}, z_{2_{max}}$ are the bounds of the two virtual state variables, then global bounded synchronization is achieved by the network.
\end{thm}

\begin{proof}
Let us consider an unweighted complete graph of $N$ HKB oscillators interconnected via a full state coupling, that is

\begin{equation*}
\dot{x}_i=
\begin{bmatrix}
x_{i_2} \\
- (\alpha_i x_{i_1}^2+\beta_i x_{i_2}^2-\gamma_i)x_{i_2} - \omega_i^2 x_{i_1}
\end{bmatrix}
\end{equation*}

\begin{equation}
-\hat{c} \sum_{j=1}^N a_{ij} \left(x_i-x_j \right), \quad \forall i \in[1,N]
\end{equation}
where $x_i \in \mathbb{R}^2$ is the state variable of node $i$ and $\hat{c} := \frac{c}{N-1}$ since each node in a connected complete graph has $N-1$ neighbors. 
The virtual system reads
\begin{equation}
\dot{z} = \begin{bmatrix}
z_2-\hat{c}Nz_1+\hat{c}\sum_{j=1}^{N} x_{j_1} \\
-\left( \tilde{\alpha}z_1^2+\tilde{\beta}z_2^2-\tilde{\gamma} \right)z_2 - \tilde{\omega}^2 z_1 - \hat{c}Nz_2 + \hat{c} \sum_{j=1}^{N} x_{j_2}
\end{bmatrix}
\end{equation}
where $z \in \mathbb{R}^2$ is the state variable of the virtual system and $\tilde{\alpha}:=\frac{1}{N}\sum_{i=1}^{N} \alpha_i$, $\tilde{\beta}:=\frac{1}{N}\sum_{i=1}^{N} \beta_i$, $\tilde{\gamma}:=\frac{1}{N}\sum_{i=1}^{N} \gamma_i$, $\tilde{\omega}:=\frac{1}{N}\sum_{i=1}^{N} \omega_i$. The Jacobian of the virtual system is:

\begin{equation}
\label{eqn:jacVS}
J(t,z) = \begin{bmatrix}
-\hat{c}N & 1 \\
-(2\tilde{\alpha}z_2z_1+\tilde{\omega}^2) & -\tilde{\alpha}z_1^2-3\tilde{\beta}z_2^2-\hat{c}N+\tilde{\gamma}
\end{bmatrix}
\end{equation}

In order to prove global bounded synchronization of the network, we need the virtual system to be contracting. In order to do so, we apply the algorithm presented in \cite{RdB09} to Eq. \ref{eqn:jacVS}.
When $i=1,j=2$, it is immediate to see that statement \emph{S1} is true, while \emph{S2} and \emph{S3} are false ($c$ might be in general time varying), leading to $|J(1,2)|>|J(1,1)|$ and $|J(2,1)|<|J(2,2)|$. When $i=2,j=1$ instead, inequalities to satisfy and CFS depend on the sign of $\tilde{\alpha}$ and $\tilde{\beta}$. Supposing without loss of generality that $\tilde{\alpha},\tilde{\beta}>0$ as usually done in literature \cite{FJ08, KdGRT09}, it is immediate to see that an inequality corresponding to the fulfilment of \emph{S1} needs to be added to the list of CFS generated by the algorithm (a worst case scenario is $-\hat{c}N+\tilde{\gamma}<0$), and that both \emph{S2} and \emph{S3} are again false, leading to the two same conditions. This means that the network achieves global bounded synchronization when the following system is satisfied:

\begin{equation*}
\begin{cases}
\hat{c}>\frac{\tilde{\gamma}}{N}\\
1>\hat{c}N\\
| 2\tilde{\alpha}z_1z_2+\tilde{\omega}^2 | < | \tilde{\alpha}z_1^2+3\tilde{\beta}z_2^2-\tilde{\gamma}+\hat{c}N |
\end{cases}
\end{equation*}

\begin{equation}
\Leftrightarrow
\begin{cases}
\frac{\tilde{\gamma}}{N}<\hat{c}<\frac{1}{N}\\
| 2\tilde{\alpha}z_1z_2+\tilde{\omega}^2 | < | \tilde{\alpha}z_1^2+3\tilde{\beta}z_2^2-\tilde{\gamma}+\hat{c}N |
\end{cases}
\end{equation}

Supposing that the dynamics of the virtual system is bounded, meaning that $|z_1(t)| \le z_{1_{max}}, |z_2(t)| \le z_{2_{max}}$ $\forall t \ge 0$, we can consider the following worst case scenario

\begin{equation}
\begin{cases}
\frac{\tilde{\gamma}}{N}<\hat{c}<\frac{1}{N}\\
2\tilde{\alpha}z_{1_{max}}z_{2_{max}}+\tilde{\omega}^2 < \hat{c}N - \tilde{\gamma}
\end{cases}
\end{equation}
which leads to

\begin{equation}
\label{eqn:coupStrenCT}
\frac{2\tilde{\alpha}z_{1_{max}}z_{2_{max}}+\tilde{\omega}^2+\tilde{\gamma}}{N}<\hat{c}<\frac{1}{N}
\end{equation}
and, as a consequence, to Eq. \ref{eqn:coupStrenCT2}.

So we can conclude that if the coupling strength $c$ fulfills Eq. \ref{eqn:coupStrenCT2}, the heterogeneous network of HKB oscillators overlying a complete graph achieves bounded synchronization.
\end{proof}

\begin{rem}
Note that when the number of nodes in the network $N$ is really high, then $\frac{N-1}{N} \rightarrow 1$. This means that global bounded synchronization is achieved when:

\begin{equation}
\label{eqn:coupStrenCT3}
2\tilde{\alpha}z_{1_{max}}z_{2_{max}}+\tilde{\omega}^2+\tilde{\gamma} < c < 1
\end{equation}

\end{rem}

\subsection{Lyapunov theory}
Let $\mathcal{D}$ be the set of diagonal matrices, $\mathcal{D}^+$ the set of positive definite diagonal matrices and $\mathcal{D}^-$ the set of negative definite diagonal matrices. Let us now define \emph{QUAD} and \emph{QUAD Affine} vector fields \cite{DLdBL14}.

\begin{dfn}
Given $n \times n$ matrices $P \in \mathcal{D}^+, W_i \in \mathcal{D}$, the vector field $f_i$ is said to be \emph{QUAD($P,W_i$)} iff
\begin{equation}
(z-w)^T P [f_i(t,z)-f_i(t,w)] \le (z-w)^T W_i (z-w)
\end{equation}
for any $z,w \in \mathbb{R}^n$ and for any $t \ge 0$.
\end{dfn}

\begin{dfn}
Given $n \times n$ matrices $P \in \mathcal{D}^+, W_i \in \mathcal{D}$, the vector field $f_i$ is said to be \emph{QUAD($P,W_i$) Affine} iff $f_i(t,x_i)=h_i(t,x_i)+g_i(t,x_i)$ and
\begin{itemize}
\item $h_i$ is QUAD($P,W_i$);
\item $\exists \ M<\infty : ||g_i(t,z)||_2<M, \ \forall z \in \mathbb{R}^n, \forall t \ge 0$
\end{itemize}
\end{dfn}

Let us consider a heterogeneous network of $N$ agents interconnected via a linear coupling:

\begin{equation}
\label{eqn:qvflc}
\dot{x}_i(t) = f_i(t,x_i)-\frac{c}{\mathcal{N}_i}\sum_{j=1}^{N}a_{ij}\Gamma(x_i-x_j), \quad c>0
\end{equation}

where $\Gamma \in \mathbb{R}^{n \times n}$. Note that this is a generalization of the full state coupling previously introduced, which can be obtained by setting $\Gamma=I_n$. As reported in \cite{DLdBL14} in details, in order to prove global bounded synchronization of a network of $N$ nonidentical QUAD Affine systems coupled via a linear interaction protocol, we need $h_i(t,x_i)$ to be QUAD($P,W_i$) with $W_i \in \mathcal{D}^-$ for all the nodes in the network at any time instant. However, in a heterogeneous network of $N$ HKB oscillators with vector fields described by Eq. \ref{eqn:hkbInternalDynamics}, regardless of the way we define $h_i$ and $g_i$ it is never possible to satisfy the following condition

\begin{equation}
(z-w)^T P [h_i(t,z)-h_i(t,w)] \le (z-w)^T W_i (z-w)
\end{equation}

with definite negative matrices $W_i$. Indeed, the right-hand term is always negative, while the left-hand one can be positive for any value of $P>0$.
On the other hand, in order to avoid conditions on the sign of the matrices $W_i$, it is necessary to write the dynamics of the nodes in the following way

\begin{equation}
f_i(t,x_i)=h_i(t,x_i)+g_i(t,x_i) \quad \forall i=1,2,...,N
\end{equation}

with $h_i(t,z)=h_j(t,z)=h(t,z) \ \forall i,j \in [1,N], \ \forall z \in \mathbb{R}^n$, and with all the terms $g_i$ being bounded, at any time instant. 

In particular, in \cite{DLdBL14} the authors formalize the following theorem.

\begin{thm}
\label{thm:qvflcthm}
Let us consider a heterogeneous network of $N$ agents interconnected via a linear coupling as described in Eq. \ref{eqn:qvflc}. Let us suppose that $f_i(t,x_i)=h(t,x_i)+g_i(t,x_i)$ and that:
\begin{enumerate}
\item the network is made up of $N$ QUAD($P,W$) Affine systems, with $P \in \mathcal{D}^+$ and $W \in \mathcal{D}$;
\item $\Gamma$ is a positive semidefinite diagonal matrix;
\item if $W$ is made up of $l \in [0,n]$ non-negative elements, which without loss of generality can be collected in its $l \times l$ upper-left block, then $\Gamma$ is made up of $\bar{l}$ positive elements, where $l \le \bar{l} \le n$, which again without loss of generality can be collected in its $\bar{l} \times \bar{l}$ upper-left block;
\item $\exists \ 0<\bar{M}<\infty : ||g_i(t,x_i)||_2<\bar{M} \ \forall i=1,2,...,N, \forall t \ge 0$.
\end{enumerate}

Then, we can claim that global bounded synchronization is achieved by the network. In particular, if we define matrix $L_{\mathcal{N}}=\{ l_{\mathcal{N}_{ij}} \}$ as

\begin{equation}
l_{\mathcal{N}_{ij}} := \begin{cases}
\frac{1}{\mathcal{N}_i} \sum_{k=1}^N a_{ik}, & \mbox{if } i=j \\
-\frac{a_{ij}}{\mathcal{N}_i}, & \mbox{if } i \neq j \mbox{ and } (i,j) \mbox{ are neighbors}\\
0, & \mbox{otherwise}
\end{cases}
\end{equation}
we can state that $\exists \  0<\bar{c}<\infty, \epsilon>0 : \lim_{t \to \infty} \eta(t) \le \epsilon \ \forall c > \bar{c}$, where

\begin{equation}
\label{QUADthMinc}
\bar{c} = \min_{P,W} \ \max \left( \frac{\lambda_M\left( W_l \right)}{\lambda_2 \left(L_{\mathcal{N}} \otimes P_l \Gamma_l \right)},0 \right)
\end{equation}
with $W_l,P_l,\Gamma_l$ representing the $l \times l$ upper-left block of matrices $W,P,\Gamma$, respectively, and where for a given value of $c>\bar{c}$

\begin{equation}
\label{QUADthErrBound}
\epsilon = \min_{P,W} \ \frac{\sqrt{N} \bar{M} ||P||_2}{-\max \left( \lambda_M\left( W_l \right) -c \lambda_2 \left(L_{\mathcal{N}} \otimes P_l \Gamma_l \right), \lambda_M \left( W_{n-l} \right) \right)}
\end{equation}
with $W_{n-l}$ representing the $(n-l) \times (n-l)$ lower-right block of matrix $W$ and with the assumption that $c \lambda_2 \left(L_{\mathcal{N}} \otimes P_l \Gamma_l \right) > \lambda_M\left( W_l \right)$.

\end{thm}

\begin{proof}
See \cite{DLdBL14}.
\end{proof}

We can thus derive the following corollary.

\begin{cor}
Let us consider a heterogeneous network of $N$ HKB oscillators interconnected via a linear coupling. Supposing that the topology of the network is simple and undirected, and assuming that $\gamma_i=\tilde{\gamma}  \ \forall i \in [1,N]$, if the coupling strength satisfies the inequality 

\begin{equation}
\label{eqn:cbarQuad}
c \ge \bar{c} 
= \min_{W(1,1),P(1,1),P(2,2)>0} \frac{\max \left( W(1,1),\tilde{\gamma} P(2,2) \right) }{\lambda_2\left( L_{\mathcal{N}} \right) \min_{j=1,2}\left( P(j,j)\Gamma(j,j) \right)} 
\end{equation}
then global bounded synchronization is achieved by the network. In particular, we can claim that

\begin{equation}
\label{eqn:enormQuad}
\lim_{t \to \infty} \eta(t) \le \epsilon
= \min_{W(1,1),P(1,1),P(2,2),d_\epsilon>0} \frac{\sqrt{N}\bar{M} \max \left(P(1,1),P(2,2) \right) }{d_\epsilon}
\end{equation}
where 
\begin{equation}
d_\epsilon:=c\lambda_2 \left( L_{\mathcal{N}} \right) \min_{j=1,2 } \left( P(j,j) \Gamma(j,j) \right)
-\max \left( W(1,1),\tilde{\gamma} P(2,2) \right)
\end{equation}
\end{cor}

\begin{proof}
First of all we need to write the dynamics of each node in the network as $f_i(t,x_i)=h(t,x_i)+g_i(t,x_i)$. This is possible if we suppose that $\gamma_i=\tilde{\gamma} \ \forall i \in [1,N]$ and define:

\begin{equation*}
h(t,x_i)=\begin{bmatrix}
0\\
\tilde{\gamma} x_{i_2}
\end{bmatrix}
\end{equation*}

\begin{equation*}
g_i(t,x_i)=\begin{bmatrix}
x_{i_2}\\
-(\alpha_i x_{i_1}^2 + \beta_i x_{i_2}^2)x_{i_2}-\omega_i^2x_{i_1}
\end{bmatrix}
\end{equation*}

Then we need to verify whether the nodes in the network are QUAD($P,W$) Affine systems. In particular, this means that we need $h$ to be QUAD($P,W$), with $P \in \mathcal{D}^+$ and $W \in \mathcal{D}$. Therefore, if we define $P=diag \{P(1,1),P(2,2) \}$ with $P(1,1),P(2,2)>0$, $W=diag \{W(1,1),W(2,2) \}$ and $h(t,z)=[0 \ \tilde{\gamma} z_2]^T \ \forall z \in \mathbb{R}^2$, we have to satisfy:

\begin{equation}
\label{eqn:hquadineq}
P(2,2) \tilde{\gamma} (z_2-w_2)^2
\le W(1,1)(z_1-w_1)^2+W(2,2)(z_2-w_2)^2
\end{equation}

Hence, if we choose $W(2,2)=\tilde{\gamma} P(2,2)$, it possible to reduce Eq. \ref{eqn:hquadineq} to

\begin{equation}
W(1,1)(z_1-w_1)^2 \ge 0
\end{equation}
which is true for any $W(1,1)>0$. This means that the first hypothesis of Theorem \ref{thm:qvflcthm} simply reduces to choosing any $P \in \mathcal{D}^+$ and $W=diag \{ W(1,1),\tilde{\gamma} P(2,2) \}$ for any $W(1,1)>0$. 

Since $W \in \mathbb{R}^{2 \times 2}$ is made up of $2$ non-negative elements, we have that $l=\bar{l}=2$. Therefore, in order to satisfy the second and the third hypotheses of Theorem \ref{thm:qvflcthm}, we need $\Gamma$ to be a diagonal positive definite matrix, that is $\Gamma \in \mathcal{D}^+$ (note that this is true when the nodes are connected through a full state coupling, since it corresponds to $\Gamma=I_2$). 

Finally, the last hypothesis to satisfy is related to the boundedness of the terms $g_i$ at any time instant. As already shown before, we have chosen:

\begin{equation}
\label{eqn:qvfgf}
g_i(t,x_i) = \begin{bmatrix}
x_{i_2}\\
-(\alpha_i x_{i_1}^2 + \beta_i x_{i_2}^2)x_{i_2}-\omega_i^2x_{i_1}
\end{bmatrix}
\end{equation}

Since the dynamics of each HKB oscillator is bounded \cite{ZATdB15}, we can define

\begin{equation*}
p_{i_{max}} := \sup_{t \ge 0} \left( |x_{i_1}(t)| \right), \ v_{i_{max}} := \sup_{t \ge 0} \left( |x_{i_2}(t)| \right)
\end{equation*}

and we define as well $p_M := \max_{i} \left( p_{i_{max}} \right)$, $v_M := \max_{i} \left( v_{i_{max}} \right)$, $\alpha_M := \max_{i} \left( |\alpha_i| \right)$, $\beta_M := \max_{i} \left( |\beta_i| \right)$, $\omega_M := \max_{i} \left( |\omega_i| \right)$. Therefore, from Eq. \ref{eqn:qvfgf} we get:

\begin{equation*}
||g_i||_2 \le |x_{i_2}| + |(\alpha_i x_{i_1}^2 + \beta_i x_{i_2}^2)x_{i_2}+\omega_i^2x_{i_1}|
\end{equation*}

\begin{equation*}
\le |x_{i_2}| + |\alpha_i x_{i_1}^2 + \beta_i x_{i_2}^2| |x_{i_2}| + \omega_i^2 |x_{i_1}|
\end{equation*}

\begin{equation}
\le (1+|\alpha_i| p_{i_{max}}^2 + |\beta_i| v_{i_{max}}^2) v_{i_{max}} + \omega_i^2 p_{i_{max}} := M_i
\end{equation}

Besides, we have that

\begin{equation}
\label{eqn:boundM}
M_i \le (1+\alpha_M p_M^2 + \beta_M v_M^2) v_M + \omega_M^2 p_M := \bar{M}
\end{equation}

This means that the fourth hypothesis of Theorem \ref{thm:qvflcthm} is always satisfied in the case of HKB oscillators, and the bound $\bar{M}$ is defined in Eq. \ref{eqn:boundM}.

In order to find an easier expression of the minimum value required for the coupling strength and of the upper-bound for the error norm, we can take advantage of the particular form of matrices $P$ and $W$:

\begin{equation*}
P = P_2 = \begin{bmatrix}
P(1,1) & 0 \\
0 & P(2,2)
\end{bmatrix}, \qquad P(1,1),P(2,2)>0
\end{equation*}

\begin{equation*}
W = W_2 = \begin{bmatrix}
W(1,1) & 0\\
0 & \tilde{\gamma}P(2,2)
\end{bmatrix}, \qquad W(1,1)>0
\end{equation*}

Therefore, from Eq. \ref{QUADthMinc} we have that the minimum value $\bar{c}$ for the coupling strength that guarantees bounded synchronization of the network is given by Eq. \ref{eqn:cbarQuad}, while for a given $c>\tilde{c}$, from Eq. \ref{QUADthErrBound} we have that the upper bound of the error norm is given by Eq. \ref{eqn:enormQuad}.

So we can conclude that if $c>\bar{c}>0$, where $\bar{c}$ is defined in Eq. \ref{eqn:cbarQuad}, global bounded synchronization is achieved.
\end{proof}

\subsection{Numerical validation}
As previously shown for a connected simple undirected heterogeneous network of $N$ HKB oscillators, by making use of contraction theory it is possible to guarantee bounded synchronization if we suppose that the underlying topology is an unweighted complete graph (all-to-all network). On the other hand, by making use of Lyapunov theory, bounded synchronization can be achieved regardless of the topology and the weights of the interconnections, although an assumption has to be made on one of the parameters of the nodes ($\gamma_i=\tilde{\gamma} \ \forall i \in [1,N]$). In order to be able to study the most possible general case, we consider a weighted random graph of $N=5$ HKB nonlinear oscillators characterized by $\gamma_i=\tilde{\gamma} \ \forall i \in [1,N]$. In particular, in such a random graph the odds of an edge connecting two nodes is worth $60 \%$ and its weight is randomly picked between $0$ and $2$ (see Fig. \ref{fig:topology}). As for the parameters and the initial conditions of the nodes, see Table \ref{table:5nodesTable}. 
Moreover, we set $T=200s$.

\begin{figure}
 \centering
 \includegraphics[width=.7\textwidth]{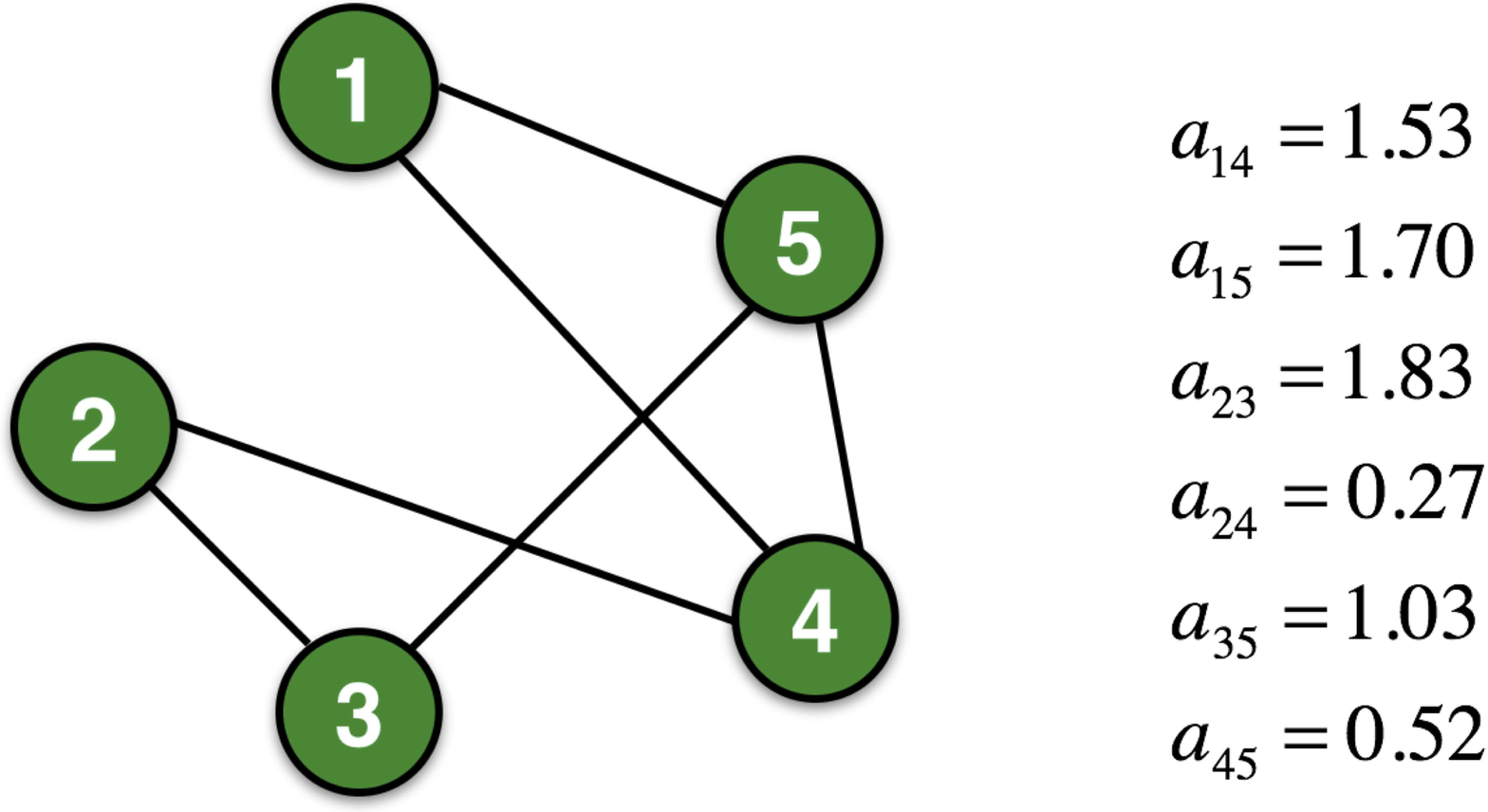}
 \caption{Underlying topology - simple connected weighted graph}
 \label{fig:topology}
\end{figure}

\begin{table}[ht]
\caption{Numerical simulations - parameters and initial values for a network of $N=5$ HKB oscillators}
\centering
\label{table:5nodesTable}
\begin{tabular}{llllll}
\hline\noalign{\smallskip}
Nodes & $\alpha_i$ & $\beta_i$  & $\gamma_i$ & $\omega_i$ & $x_i(0)$ \\
\noalign{\smallskip}\hline\noalign{\smallskip}
Node 1 & $0.46$ & $1.16$  & $0.58$ & $0.16$ & $[-1.4, +0.3]$  \\
Node 2 & $0.37$ & $1.20$  & $0.58$ & $0.26$ & $[+1.0, +0.2]$ \\
Node 3 & $0.34$ & $1.73$  & $0.58$ & $0.18$ & $[-1.8, -0.3]$  \\
Node 4 & $0.17$ & $0.31$  & $0.58$ & $0.21$ & $[+0.2, -0.2]$ \\
Node 5 & $0.76$ & $0.76$  & $0.58$ & $0.27$ & $[+1.5, +0.1]$\\
\noalign{\smallskip}\hline
\end{tabular}
\end{table} 

This scenario leads to $p_M=2.6$, $v_M=0.96$, $\bar{M}=7.6$, $\lambda_2 \left(L_{\mathcal{N}} \right)=0.4112$, $P(1,1)=0.077$, $P(2,2)=0.077$, $W(1,1)=0.001$, $W(2,2)=0.045$ and $\bar{c} = 1.4211$. Fig. \ref{fig:FSC} shows $x_1(t)$ for all the nodes in the network (blue line: node 1, green line: node 2, red line: node 3, cyan line: node 4, magenta line: node 5) when connected through a full state coupling protocol with $c=1.45$: as we can see, our analytical results are confirmed and synchronization is achieved by the network.

On the other hand, in Fig. \ref{fig:eta} we show that bounded synchronization can actually be achieved for smaller values of the coupling strength when considering a full state coupling ($c=0.07$), and that it can be achieved also with the two other coupling protocols presented earlier in this paper ($c_1=c_2=0.1$ for the partial state coupling and $a=b=-1, c=0.1$ for the HKB coupling, respectively). Indeed, with this choice of coupling strength, the error norm $\eta(t)$ is roughly bounded by $\epsilon \simeq 2$. In Fig. \ref{fig:gsCompNum} we then show the trend of the group synchronization obtained respectively in all the three cases: as we can see, after an initial transient $\rho_g(t)$ reaches a much higher value, confirming what observed in \cite{RGFGM12}. In particular, the trend obtained when considering an HKB coupling resembles the most the one obtained in the real experiments involving human people: indeed, in this case the group synchronization presents persistent oscillations with the highest amplitude, as observed in the rocking chairs experiments. 

\begin{figure}
 \centering
 \subfloat[first state variable]{\includegraphics[width=.5\textwidth]{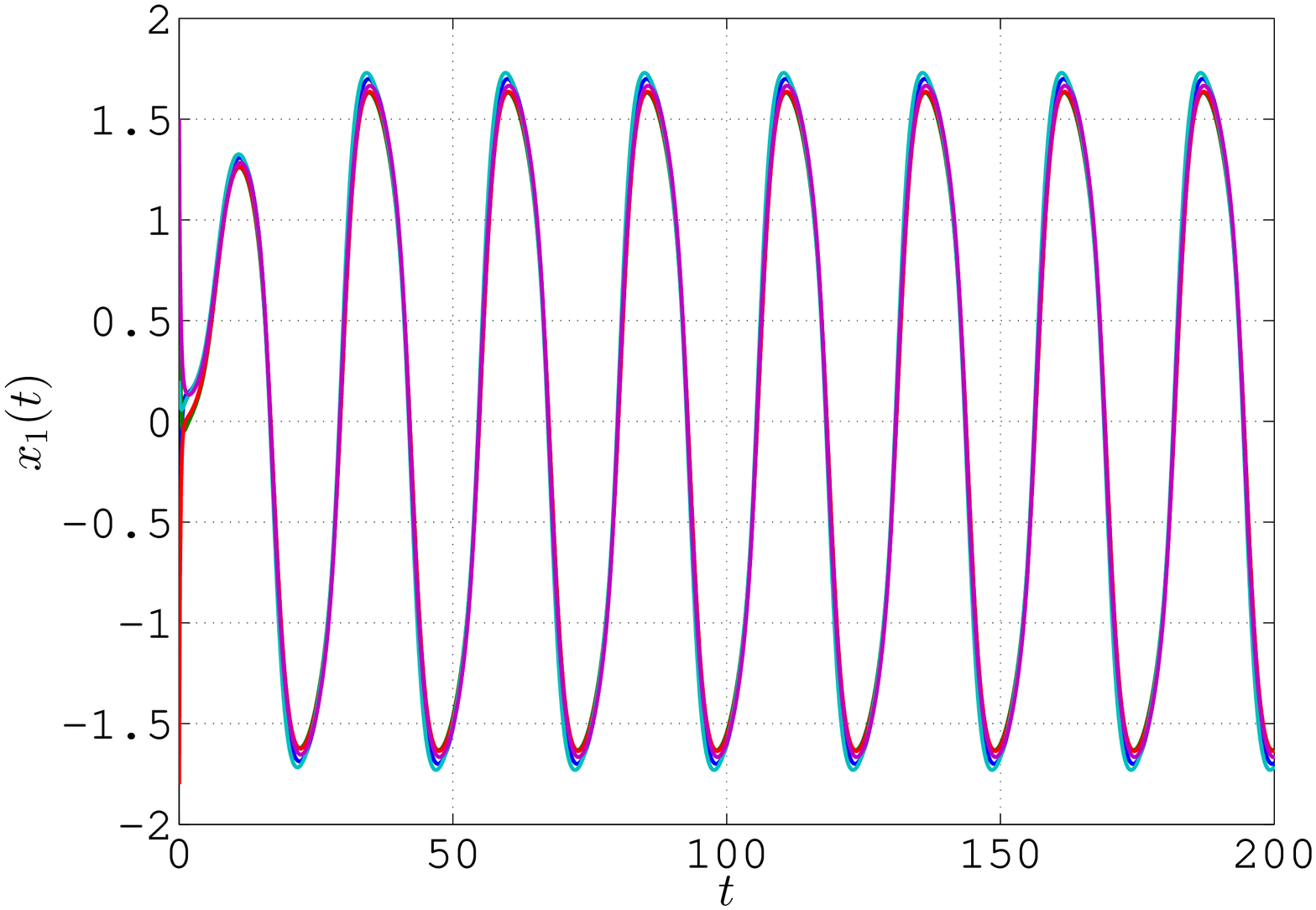}}
 \subfloat[first state variable - zoom]{\includegraphics[width=.5\textwidth]{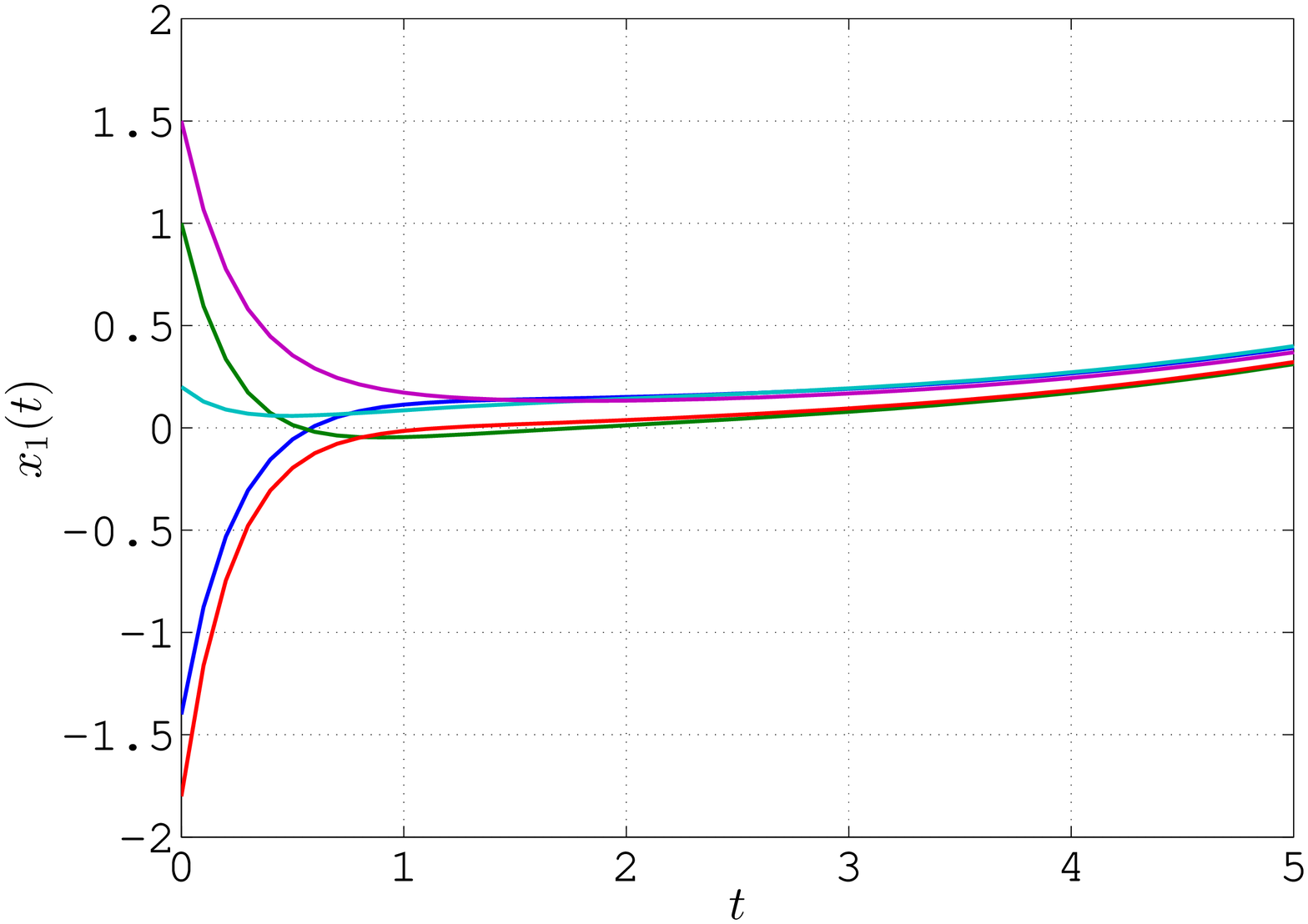}}
 \caption{First state variable $x_{i_1}$ in a simple connected weighted heterogeneous network of $N=5$ HKB oscillators - full state coupling ($c=1.45$)}
 \label{fig:FSC}
\end{figure}

\begin{figure}
 \centering
 \includegraphics[width=.7\textwidth]{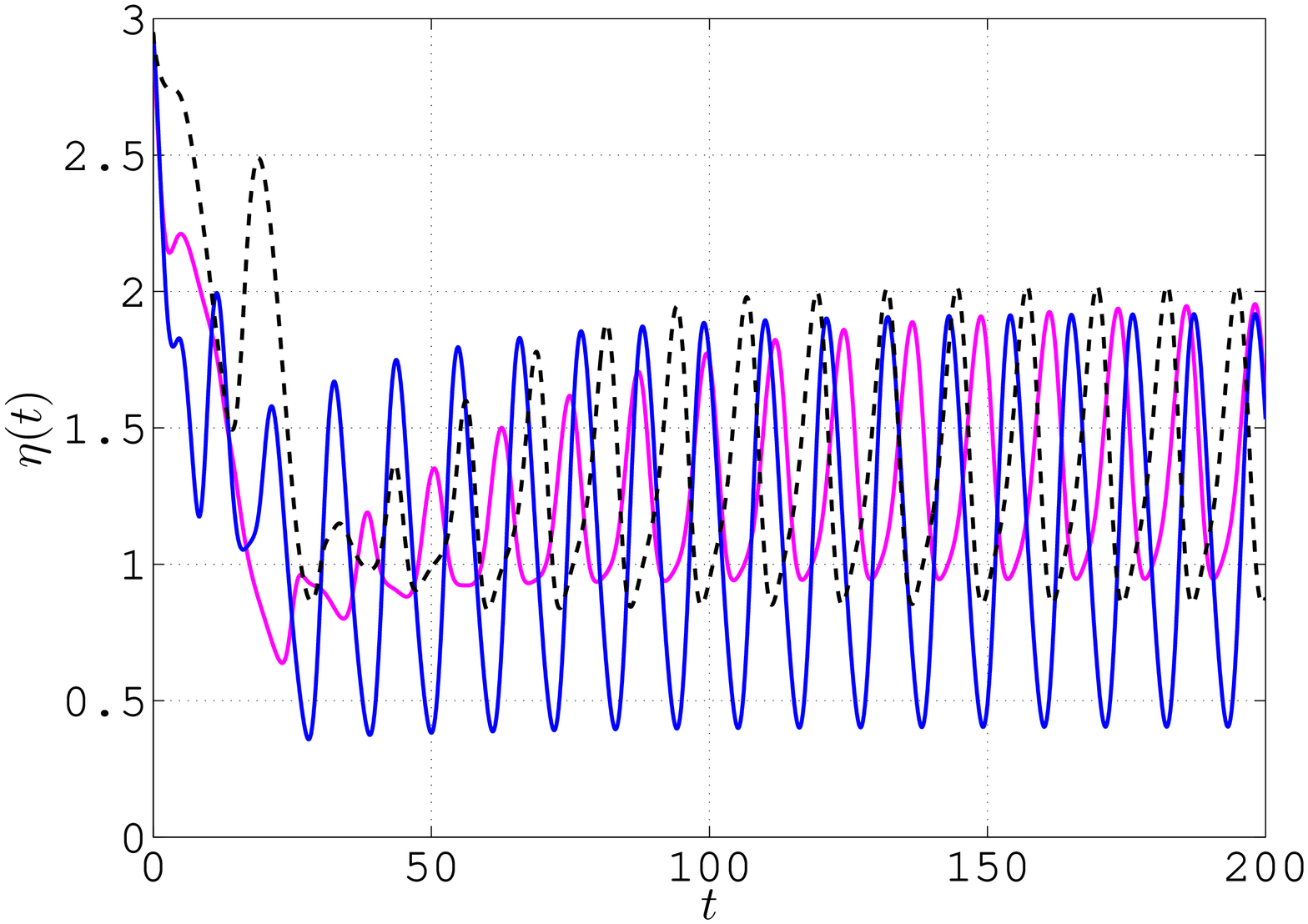}
 \caption{Error norm in a simple connected weighted heterogeneous network of $N=5$ HKB oscillators - magenta solid line: full state coupling ($c=0.07$), blue solid line: partial state coupling ($c_1=c_2=0.1$), black dashed line: HKB coupling ($a=b=-1, c=0.1$)}
 \label{fig:eta}
\end{figure}

\begin{figure}
 \centering
 \subfloat[group synchronization]{\includegraphics[width=.5\textwidth]{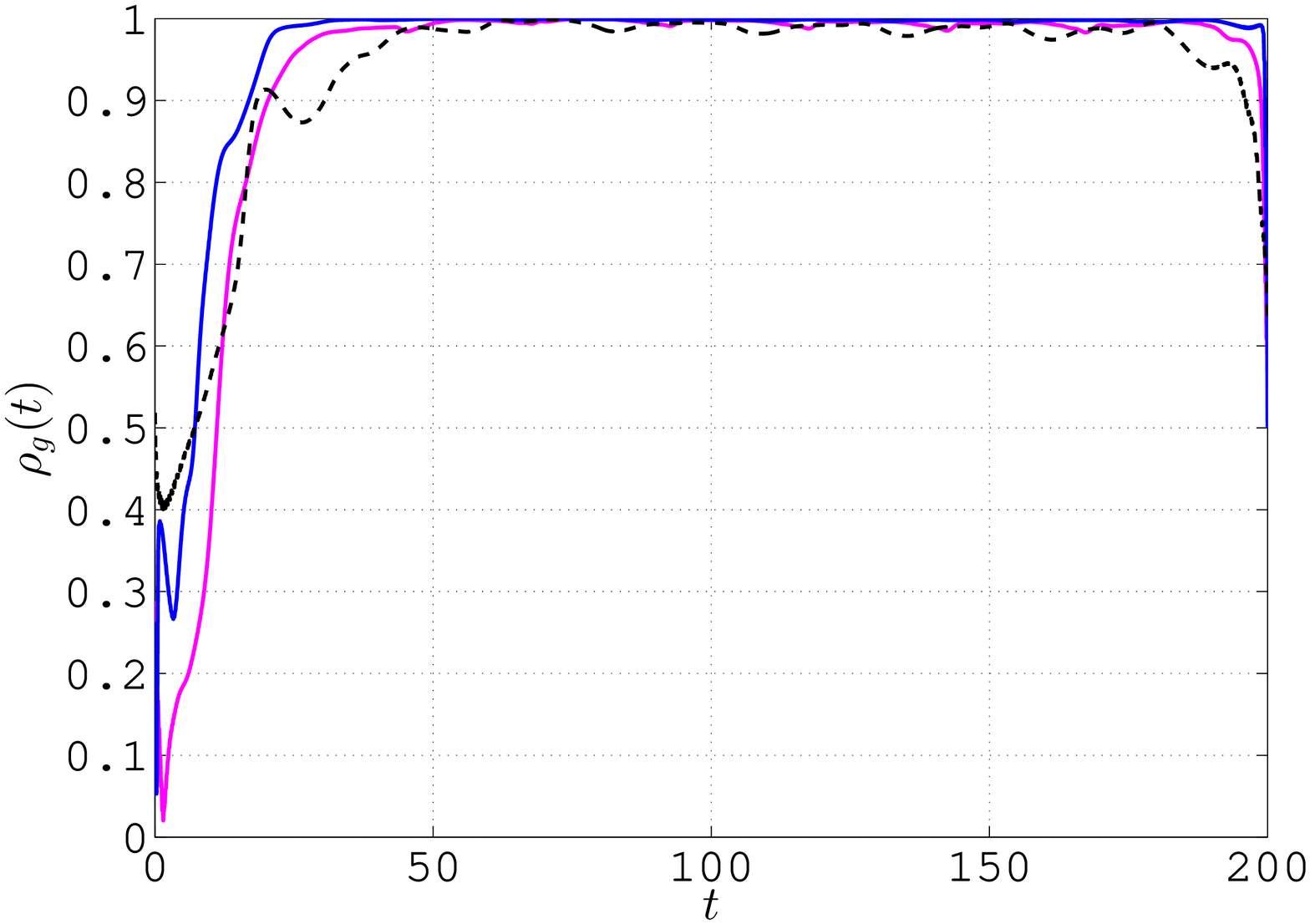}}
 \subfloat[group synchronization - zoom]{\includegraphics[width=.5\textwidth]{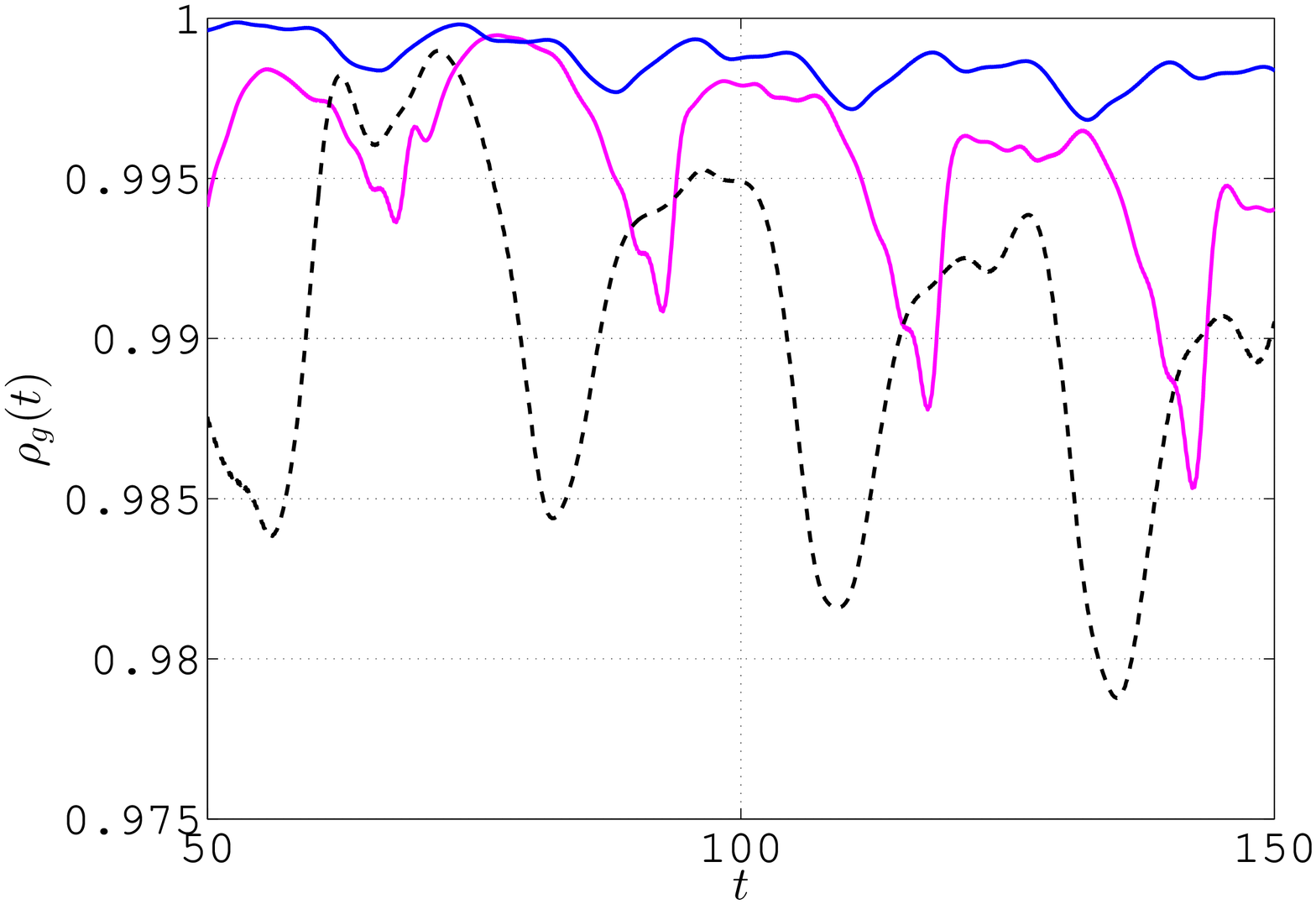}}
 \caption{Group synchronization in a simple connected weighted heterogeneous network of $N=5$ HKB oscillators - magenta solid line: full state coupling ($c=0.07$), blue solid line: partial state coupling ($c_1=c_2=0.1$), black dashed line: HKB coupling ($a=b=-1, c=0.1$)}
 \label{fig:gsCompNum}
\end{figure}

\section{Conclusion}
\label{sec:conclusion}
We have proposed a mathematical model for movement synchronization of a group of three or more people. In particular we have considered heterogeneous networks of HKB nonlinear oscillators, in which each equation is characterized by a different set of parameters to account for human-to-human variability. Three different coupling models, both linear and nonlinear, have been investigated, and the effects of adding an external entrainment signal have been analyzed.
We have found analytical conditions for a connected simple undirected network to achieve bounded synchronization when considering a full state coupling as interaction protocol among the nodes, while we have numerically shown that bounded synchronization can be achieved also when considering a partial state coupling or a HKB coupling.
In particular, we have observed that it is possible to replicate some of the synchronization features obtained in the rocking chairs experiments with all the three coupling protocols proposed in this paper, although the most realistic one is achieved when connecting the nodes through a nonlinear HKB coupling. Indeed, in this case the group synchronization presents persistent oscillations with the highest amplitude, as observed in \cite{RGFGM12}.

Some viable extensions of this work include performing real experiments involving hand synchronization of more than two players and then choosing the coupling protocol that can best capture the observations coming from such experiments and explain the onset and features of movement coordination among the human players.  Finally, we wish to prove global bounded synchronization with any kind of coupling protocol, both linear and nonlinear, also in the more general case of directed networks.

\section*{Acknowledgements}
The authors wish to acknowledge support from the European Project AlterEgo FP7 ICT 2.9 - Cognitive Sciences and Robotics, Grant Number 600610.

\end{document}